\documentclass[reqno]{amsart}
\usepackage{mathrsfs}
\usepackage{amsmath}
\usepackage{amsthm}
\usepackage{amsfonts}
\usepackage{dsfont} %makes \mathds{1} work as an indicator
\usepackage{amssymb}
\usepackage{url}
\usepackage{enumerate}
\usepackage[pdftex,bookmarks=true]{hyperref}
\usepackage[usenames, dvipsnames]{color}
\usepackage{verbatim}
\usepackage{tikz}
\usepackage{subcaption}
\usepackage{pgfplots}
\pgfplotsset{compat=1.18}
\usepackage{orcidlink}

\usepackage{adjustbox}
\usepackage{xcolor}
\usepackage{pdfsync}
\usepackage[magyar, english]{babel}%magyar is so that the \H{o} in the bibliography loads correctly
\usepackage{t1enc}

\usepackage{todonotes}
\newcommand\R{{\mathbb{R}}}

\renewcommand\P{{\mathbf{P}}}
\newcommand\E{{\mathbf{E}}}

\newcommand\diag{{\operatorname{diag}}}

\newcommand\dist{{\operatorname{dist}}}
\newcommand\Z{{\mathbb{Z}}}

\newcommand\col{{\mathbf{c}}}
\newcommand\row{{\mathbf{r}}}

\newcommand\var{{ \operatorname{Var}}}

\newcommand\al{\alpha}
\newcommand\la{\lambda}
\newcommand\1{\mathbf{1}}

\newcommand\Bc{{\mathbf c}}
\newcommand\Bd{{\mathbf d}}

\newcommand\Bf{{\mathbf f}}

\newcommand\Bu{{\mathbf u}}
\newcommand\Bv{{\mathbf v}}
\newcommand\Bw{{\mathbf w}}
\newcommand\Bx{{\mathbf x}}
\newcommand\By{{\mathbf y}}

\newcommand\BN{{\mathbf N}}

\renewcommand\P{{\mathbb P }}

\newcommand\CA{{\mathcal A}}

\newcommand\CE{{\mathcal E}}

\newcommand\CN{{\mathcal N}}

\newcommand\LCD{\mathbf{LCD}}

\newcommand\supp{\operatorname{supp}}

\newcommand\eps{\varepsilon}

\newcommand\bs{\backslash}

\newcommand\bq{\begin{equation}}
\newcommand\eq{\end{equation}}

\newcommand\wb{\overline}

\usepackage{fullpage}

\theoremstyle{plain}

\newtheorem{prop}{Proposition}[section]

\newtheorem{theorem}[prop]{Theorem}
\newtheorem{corollary}[prop]{Corollary}

\newtheorem{lemma}[prop]{Lemma}
\newtheorem{fact}[prop]{Fact}

\newtheorem{remark}[prop]{Remark}
\newtheorem{claim}[prop]{Claim}

\theoremstyle{definition}
\newtheorem{definition}[subsection]{Definition}

\begin{document}

\title{Eigenvalue Gaps of the Laplacian of Random Graphs}

\author{Nicholas J. Christoffersen\orcidlink{0000-0003-4893-8615}}
\address{Department of Mathematics\\ University of Colorado Boulder \\ Campus Box 395 \\ Boulder, CO  80309 USA}
\email{nicholas.christoffersen@colorado.edu}

\author{Kyle Luh\orcidlink{0000-0002-1822-3443}}
\address{Department of Mathematics\\ University of Colorado Boulder \\ Campus Box 395 \\ Boulder, CO  80309 USA}
\email{kyle.luh@colorado.edu}
\thanks{K. Luh has been partially supported by Simons Grant MP-TSM-00001988.}

 \author{Hoi H. Nguyen\orcidlink{0000-0002-4590-9301}}
 \address{Department of Mathematics\\ The Ohio State University \\ 231 W 18th Ave \\ Columbus, OH  43210 USA}
 \email{nguyen.1261@math.osu.edu}
\thanks{Part of this work was completed while the third author was supported by a Simons Fellowship.}

\author{Jingheng Wang}
\address{Department of Mathematics\\ The Ohio State University \\ 231 W 18th Ave \\ Columbus, OH  43210 USA}
\email{wang.14053@buckeyemail.osu.edu}

\keywords{Eigenvalue Gaps, Spectral Gaps, Graph Laplacian, Quantum Random Walks, Random Matrices, Delocalization of Eigenvectors}

\subjclass[2020]{60B20, 60C05, 05C80, 15B52}
%\pacs[MSC Classification]{60B20, 60C05, 05C80, 15B52}

\begin{abstract} We show that, with very high probability, the random graph Laplacian has  simple spectrum. Our method provides a quantitatively effective estimate of the spectral gaps. Along the way, we establish results on affine no-gaps delocalization, no-structure delocalization, overcrowding and the presence of small entries in the eigenvectors of the Laplacian model. These findings are of independent interest.
\end{abstract}

\maketitle

\section{Introduction}\label{sect:intro}

Let $G_{n}=G(n,p)$ be a random Erd\H{o}s-R\'enyi graph with a fixed parameter $p$, and let $A_{n}$ be its adjacency matrix. Partly motivated by a question posed by Babai (related to the well-known graph isomorphism problem in theoretical computer science), it has been shown in \cite{TV-simple, NTV} that with high probability, the spectrum $\la_{1}(A_{n}) \le \dots \le \la_{n}(A_{n})$ of $A_{n}$ is simple. More precisely,
\begin{theorem}\cite[Theorem 2.7]{NTV}\label{thm:gap}   Let $0<p<1$ be fixed, and let $A$ be any constant. For any $\delta > n^{-A}$, we have
$$ \max_{1\le i\le n-1} \P(|\la_{i+1}(A_{n}) -\la_{i}(A_{n})|\le \delta n^{-1/2}) =O(n^{o(1)} \delta).$$
\end{theorem}
This result was extended to sparse graphs in \cite{LV,LL}, and to Wigner matrices (with i.i.d.~subgaussian entries of mean zero and variance one) in \cite{CJMS-spt} with better bounds. 

From a random matrix perspective, gap sizes were, in fact, the original motivation for introducing random matrices in physics. Further, gap statistics are a staple of the field as can be seen in \cite{Wigner-SC, Bai-book, AGZ-book,EYY-bulk, tao-gap, EKYY-gaps} and the references therein.  More specifically, the minimal gap in Wigner matrices has also received much attention \cite{Vinson-spacing, BAB-extreme, FTW-gaps, bourgade-extreme}, where, in the latest work, it was shown to be universal up to some appropriate scaling.

Although these results are quite satisfactory, many interesting open questions remain. These include obtaining the optimal quadratic repulsion probability bound $O(\delta^{2})$ (under the assumption $\delta \ge n^{{-\omega(1)}}$) in Theorem \ref{thm:gap}, as well as extending results to other models of random matrices with dependent entries. The goal of this paper is to pursue the second direction for a natural model of random matrices that existing results and techniques do not seem to cover -- the Laplacian matrix. 

Given a graph $G_{n}$ with $n$ vertices, the (combinatorial) Laplacian of $G_{n}$, denoted by $L_{n}$, is the matrix 
$$L_{n}= D_{n}-A_{n},$$
where $D_{n}=\diag(d_{ii})_{1\le i\le n}$ is the diagonal matrix of the vertex degrees, $d_{ii} = \sum_{j=1}^{n} a_{ij}$.

The graph Laplacian is a discrete analogue of the Laplacian operator on manifolds. This matrix is a key object in spectral graph theory and is often more relevant than the adjacency matrix. The spectrum of this matrix can reveal expansion properties, isoperimetric bounds, random walk mixing times, and even the number of spanning trees of a graph \cite{godsil2001algebraic}.  Spectral algorithms involving the Laplacian have been spectacularly successful in fundamental tasks such as graph partitioning \cite{HRMP-partitioning}, clustering \cite{hagen1992new}, graph drawing \cite{fowler1994molecular}, parallel computation \cite{simon1991partitioning}, graph synchronization \cite{Band}, and numerous others \cite{Mohar-CO}. Laplacians arise in diverse fields ranging from physics to topological data analysis \cite{fyodorov-reactance, wei2023persistent, beineke2004topics}.

For the Laplacian operator on manifolds, the eigenvalues and properties of eigenfunctions have been heavily studied in the literature; see, for instance, \cite{Uhlenbeck, CZY-gaps, ChavelRG, One, LWang} and the references therein. One of the most basic questions concerns whether the eigenvalues are distinct and whether the eigenfunctions are Morse functions. Roughly speaking, it has been shown that for a generic metric, the eigenvalues are simple and the eigenfunctions are smooth. Along these lines, genericity of simple eigenvalues for a metric graph was shown in \cite{Friedlander}. Simplicity of eigenvalues and non-vanishing of eigenfunctions of a quantum graph are considered in \cite{Berkolaiko}.

For (actual) graphs, the ``spectral gap'' of the Laplacian matrix is the difference between the second smallest eigenvalue and the smallest eigenvalue (which is necessarily zero).  This important quantity is tied to the expansion property of the graph and is of interest in many combinatorial optimization problems \cite{Lubotzky-expanders}.  However, there are many settings in which the gaps between other consecutive eigenvalues play an important role.  In general, the minimum gap size is closely tied to the numerical stability of many spectral algorithms \cite{GvL-Matrix, Demmel-NLA}.  For the Laplacian in particular, the minimum gap size or the $k$-th gap size is a crucial parameter and has a specific meaning in various applications. The gap $\lambda_{n-k} - \lambda_{n-k+1}$ indicates the stability of partitioning the graph into $k$ clusters \cite{AEGL-stability}. 

Next, we highlight a key motivation for our work: quantum random walks on graphs. Unlike classical random walks, where mixing times are governed by the spectral gap of the transition matrix, the mixing time of quantum random walks is determined by the smallest eigenvalue gap of the Laplacian or the adjacency matrix \cite{AAKV-qrw, kempe2003quantum, CLR-quantumwalks}. These quantum walks serve as the foundation for efficient quantum algorithms that often significantly outperform their classical counterparts. Many of these quantum algorithms implicitly assume that the spectrum of the graph Laplacian is simple \cite{AAKV-qrw, CLR-analog, childs2003exponential, childs2004quantum}. More specifically, let $G$ be a graph with vertex set $V = \{1, \dots, n\}$. The Hamiltonian governing the quantum walk on $G$ is given by $\gamma L_G$, where $L_G$ is the graph Laplacian and $\gamma$ is a scaling factor. Given an initial state $|\psi_0\rangle \in \mathcal{H}$, the state of the walker at time $t$ evolves as $|\psi(t)\rangle = e^{-i \gamma L_G t} |\psi_0\rangle$. The probability that the walker is at node $|f\rangle$ after time $T$ is  
\[
P_f(T) = \frac{1}{T} \int_0^T |\langle f | e^{-i \gamma L_G t} |\psi_0\rangle|^2 dt.
\]  
The limiting distribution is given by $P_f(T \to \infty) := \lim_{T \to \infty} P_f(T) = \sum_{i=1}^{n} |\langle f | \Bv_i \rangle \langle \Bv_i | \psi_0 \rangle|^2$, where $\Bv_i$ are the eigenvectors of $L_G$. It can be shown that the discrepancy in the $L^1$-norm is  
\[
D(P_T) = \sum_{f} \left| P_f(T) - P_f(T \to \infty) \right| \leq \sum_{i \neq l} \frac{2 |\langle \Bv_i | \psi_0 \rangle| \cdot |\langle \Bv_l | \psi_0 \rangle|}{T |\lambda_i - \lambda_l|}.
\]  

Thus, controlling the sum  $\sum_{i \neq l} \frac{1}{|\lambda_i - \lambda_l|}$ is crucial, as it quantifies the separation of the entire spectrum. 

Beyond the applications above, the minimal eigenvalue gap of the Laplacian matrix also plays a significant role in many graph-based learning tasks. In particular, graph neural networks (GNNs) have become the industry standard for graph-based machine learning tasks \cite{zhou2020graph}. However, highly symmetric graphs often lead to slow convergence rates and low accuracy \cite{li2022expressive}. A promising solution is to enrich vertex features to break symmetries, often by embedding the graph into Euclidean space using the Laplacian matrix \cite{abboud2020surprising, li2020distance}. In these algorithms, the rate of convergence is determined by the size of the smallest eigenvalue gap of the Laplacian matrix \cite{wang2022equivariant, zou2020graph, ma2023laplacian, GBR-GNN}. Further applications of minimal eigenvalue gaps can be found in graph synchronization \cite{nyberg2014laplacian, Nyberg-geographs}.  

Motivated by all these directions, we show the following main result.

\begin{theorem}[Eigenvalue Gaps of Laplacian]\label{thm:gap:lap} Let $0<p<1$ be fixed and let $A$ be any constant. Let $G_{n}=G(n,p)$ be a random Erd\H{o}s-R\'enyi graph with parameter $p$, and let $L_{n}$ be the Laplacian of $G_{n}$. Then for any $\delta \ge n^{-A}$ we have
$$\max_{1\le i\le n-1} \P(|\la_{i+1}(L_{n}) -\la_{i}(L_{n})|\le \delta n^{-1/2}) =O(n^{o(1)} \delta),$$
where the implied constants are allowed to depend on $p$ and $A$.
\end{theorem}

In particular, the spectrum is simple with high probability. This provides a theoretical justification for the empirical success or implicit assumptions in the cited applications. More quantitatively, by taking the union bound over all $i$-th gaps, we find that, with probability $1-o(1)$, the minimum gap is at least of order $n^{{-3/2}-o(1)}$.

{\bf Notations.} Throughout this paper, we regard $n$ as an asymptotic parameter tending to infinity (in particular, we implicitly assume that $n$ is larger than any fixed constant, as our claims are all trivial for fixed $n$), and allow all mathematical objects in the paper to implicitly depend on $n$ unless they are explicitly declared to be ``fixed'' or ``constant''. We write $X = O(Y), X \ll Y$, or $Y=\Omega(X), Y \gg X$ to denote the claim that $|X| \le CY$ for some fixed constant $C$. We write $X=\Theta(Y)$ if $X\gg Y$ and $Y \gg X$. We also use $o(1)$ (and $\omega(1)$ respectively) to denote positive quantities that tend to zero (infinity) as $n \to \infty$.

For a square matrix $X$ of size $n$ and a number $\la$, for brevity, we write $X- \la$ instead of $X-\la I_{n}$. We use $\row_{i}(X)$ and $\col_{i}(X)$ to denote the $i$-th row and column of $X$, respectively. 
For historical reasons, for an $n \times k$ or \( k \times n \) matrix $X$ with \( k \leq n \), we let $0 \leq \sigma_k(X) \leq \cdots \leq \sigma_1(X)$ be the singular values of $X$. Conversely, for an $n \times n$ symmetric matrix $X$, we let $\lambda_1(X) \leq \cdots \leq \lambda_n(X)$ denote its eigenvalues and $\delta_i(X) = \lambda_{i+1}(X) - \lambda_{i}(X)$ be its $i^{\text{th}}$ eigenvalue gap.

We denote the discrete interval \( [m] = \left\{1, 2, \cdots, m \right\}\). For \(\Bv\in \R^{n} \) and \( I \subset [n] \), we denote \( \Bv_{I}\in \R^{I} \) the vector with entries \( v_i, \text{ for } i \in I \). Finally, throughout this paper, if not specified, any norm that appears is the standard Euclidean norm.

\section{Additional Main Results and Our Approach}\label{sect:others} 

\subsection{Eigenvalue Gaps for the Centered Model} It is also natural to study the centered Laplacian (also sometimes referred to as the Laplacian matrix or a Laplacian-type matrix)
 $$\wb{L}_{n} := L_{n} - \E L_{n}.$$
In a more general setting, we can start from a Wigner matrix $X_{n}=(x_{ij})_{1\le i,j\le n}$ (with i.i.d.~entries $x_{ij}, 1\le i<j\le n$ of mean zero and variance one) and form a centered Laplacian $\wb{X}_{n}$ from it by letting $\wb{x}_{ii} = \sum_{j, j\neq i} x_{ij}$ and $\wb{x}_{ij}=-x_{ij}$. Thus, the matrix $\wb{L}_{n}/\sqrt{p(1-p)}$ belongs to this family of random matrices.

It is known (see \cite{DJ-laplace}) that the empirical spectral distribution of a random centered Laplacian converges weakly to the free convolution of the semicircular law and the standard Gaussian, $\rho_{sc} \boxplus \BN(0,1)$. The gaps between extreme eigenvalues of the Laplacian have been studied \cite{J-low, CO-Gnp, campbell2022extreme}. Furthermore, among other things, it was shown in \cite{HL-Laplacian} that individual eigenvalue gaps in the bulk of the spectrum converge to the gap behavior of the GOE. However, these asymptotic results do not cover the basic question of whether $\wb{L}_{n}$ has simple spectrum with high probability. In this paper, we show the following result.
 
\begin{theorem}\label{thm:gap:lap'} Theorem \ref{thm:gap:lap} holds for the centered Laplacian model $\wb{L}_{n}$.
\end{theorem}

Next, we outline an approach \footnote{There are more advanced approaches to this problem, see for instance \cite[Eqn. 36]{TV-acta}.} from \cite{NTV} to prove Theorem \ref{thm:gap}, and highlight the challenges in analyzing the current model. Assume that $X_{n}$ is a symmetric/Hermitian matrix where $\la_{i+1} - \la_{i} \le \delta/n^{1/2}$.

Write
\begin{equation}\label{discussion:mn-split}
 X_n = \begin{pmatrix} X_{n-1} & \col_{n} \\ \col_{n}^* & x_{nn} \end{pmatrix},
\end{equation}
where $\col_{n}$ is the last column of \( X_{n} \) (without the final entry). From the Cauchy interlacing law, we observe that $\lambda_i(X_n) \le \lambda_i(X_{n-1})\le \lambda_{i+1}(X_n)$. Let $\Bu$ be a (unit) eigenvector of \( X_{n} \) with associated eigenvalue $\lambda_i (X_n) $. We write 
$\Bu= (\Bw,b)$, where $\Bw$ is a vector of length $n-1$ and $b$ is a scalar.  We have 
$$ \begin{pmatrix} X_{n-1} & \col_{n} \\ \col_{n}^{}* & x_{nn} \end{pmatrix} \begin{pmatrix} \Bw \\ b \end{pmatrix} = \lambda_i(X_n) \begin{pmatrix} \Bw \\ b \end{pmatrix} .$$
Extracting the top $n-1$ components of this equation, we obtain
$$ (X_{n-1}-\lambda_i(X_n)) \Bw + b \col_{n} = 0.$$
Let  $\Bv $ be a unit eigenvector  of $X_{n-1}$ corresponding to  $\lambda_i(X_{n-1})$. By multiplying the above equation on the left by $\Bv^T$, we get that
\begin{equation}\label{eqn:discussion:eig}
|b \Bv^T \col_{n} | = |\Bv^T  (X_{n-1}-\lambda_i(X_n)) \Bw| =|\lambda_i(X_{n-1})-\lambda_i(X_n)| |\Bv^T \Bw| .
\end{equation}

We conclude that if $\la_{i+1} - \la_{i} \le \delta n^{-1/2}$ holds, then  $|b \Bv^T \col_{n}| \le \delta n^{-1/2} $. If we assume that $|b|\gg n^{-1/2+o(1)}$, then
\begin{equation}\label{eqn:vX:smallball} 
|\Bv^T  \col_{n} | \ll n^{o(1)}\delta.
\end{equation}
Thus, we have reduced the problem to bounding  the probability with respect to $X_{n-1}$ and $\col_{n}$ that $X_{n-1}$ has an eigenvector $\Bv$ such that $|\Bv^T \col_{n}| = O(n^{o(1)}\delta)$. The plan is then divided into two separate steps: 
\begin{enumerate}
\item Step 1. With extremely high probability with respect to $X_{n-1}$, all of the unit eigenvectors $\Bv$ of $X_{n-1}$ do not have structure (where we will delay the precise definition of this structure to a later discussion); 
\vskip .01in
\item Step 2. Conditioned on $X_{n-1}$ from the first step, the probability with respect to $\col_{n}$ that  \eqref{eqn:vX:smallball} holds is bounded by $O(\delta)$.
\end{enumerate}

The structure of the Laplacian matrix $L_{n}$ (and $\wb{L}_{n}$) pose several challenges when compared with that of the symmetric/Hermitian matrix $A_{n}$ (which plays the role of $X_{n}$ above). The most apparent difficulty is the dependence among the entries above the diagonal. For instance if we fix all the entries of $L_{n-1}$, the upper left principal minor of size $n-1$ of $L_{n}$, then $L_{n}$ is determined. In other words, even if one can execute Step 1 for $L_{n-1}$, we cannot rely on Step 2 because $\col_{n}(L_{n})$ is already determined after fixing $L_{n-1}$. To avoid this deadlock, we will proceed as follows (using the model $G(n,1/2)$ as an example -- the approach naturally extends to $G(n,p)$ and beyond). First, assume that the vertices of $G=G_{n}$ are ordered as $(v_1,\dots, v_n)$. We  sample the entries $a_{ij}$ of the adjacency matrix (via $G(n,1/2)$) and compute the degrees $d_1,\dots, d_{n}$ of the vertices and then form $\wb{L}_n$. Let $\wb{L}_{n}'$ be the matrix obtained from $L_{n}$ by fixing all rows and columns, except that the last two columns $\col_{n-1}, \col_{n}$ (as well as the last two rows $\row_{n-1}, \row_{n}$) are replaced by $\col_{n-1}+\col_{n}$ and $\col_{n} - \col_{n-1}$ \footnote{In fact, to preserve the spectrum, we will replace $\col_{n-1}, \col_{n}$ by  $(\col_{n-1}+\col_{n})/\sqrt{2}$ and $(\col_{n} - \col_{n-1})/\sqrt{2}$ (see Section \ref{section:gap}), but let's ignore this minor problem for now as the proofs remain the same.}. 

Next, we reshuffle the neighbors of $v_{n-1}$ and $v_{n}$ as follows: consider the set $I_{n-2}$ of indices $1\le j\le n-2$ where $v_j$ is connected to {\it exactly one} of $v_{n-1}$ or $v_{n}$. Then for each $j \in I_{n-2}$, 
we flip a fair coin to either keep or swap whether each of $(v_j,v_{n-1})$ and $(v_j,v_{n})$ is an edge (see Figure \ref{figure:swap}).

In other words, if $X=(x_1,\dots,x_{n-2})$ and $Y=(y_1,\dots, y_{n-2})$ are the (restricted to the first $n-2$ coordinates) column vectors of $A_{n}$ (or $L_{n}, \wb{L}_{n}$) associated to $v_{n-1}$ and $v_{n}$, then $I_{n-2}$ is the collection of indices $j$ where $(x_j,y_j) =(0,1)$ or $(1,0)$ (respectively  $(0,-1)$ or $(-1,0)$ in the $L_{n}$ case, and $(1/2,-1/2)$ or $(-1/2,1/2)$ in the $\wb{L}_{n}$ case). We then flip a fair coin to reassign $(x_j,y_j)$ to $(0,1)$ or $(1,0)$. To some extent, this reshuffling is similar to \cite{Mckay} and \cite{Cook,Tik, NgW, NgW2} where shufflings/switchings were used within random graphs and random matrices. However, our implementation here is rather straightforward. We can easily see that the above process does not change the law of $\wb{L}_{n}$. In terms of $\wb{L}_{n}'$, this process searches for the index set $I_{n-2} \subset [n-2]$ where the entries of $\col_{n-1}(\wb{L}'_{n})$ (or the entries of $\col_{n-1}(\wb{L}_{n})+\col_{n}(\wb{L}_{n})$) are zero, and then for each $j\in I_{n-2}$, the entries of $\col_{n}(\wb{L}'_{n})$ take values $\pm 1$ (independently from each other, and from $\wb{L}'_{n-1}$) with probability 1/2. By Chernoff's bound, with high probability, $I_{n-2}$ has size approximately $(n-2)/2$ (or $(n-2)p$ in general). With the new randomness of $\col_{n}(\wb{L}_{n}')_{|_{I_{n-2}}}$, we can now move to Step 2 to study  \eqref{eqn:vX:smallball}, which can be rewritten as 
\begin{equation}\label{eqn:vX:smallball'} 
|\Bv^T_{|_{I_{n-2}}} \col_{n}(\wb{L}_{n}')_{|_{I_{n-2}}} -a | \ll \delta
\end{equation}
for some deterministic $a$.

While we can proceed with Step 2, a new problem arises: in Step 1 it is not enough to know that $\Bv$ does not have structure, but we have to show that $\Bv_{I}$ does not have structure for any eigenvector of $\wb{L}_{n-1}'$. To the authors' best knowledge, the property of eigenvectors not having local structure has not been studied before in the literature, except in the no-gaps delocalization aspect. (Although in the finite field setting, this local aspect has been studied in \cite{NgW} for null vectors.)

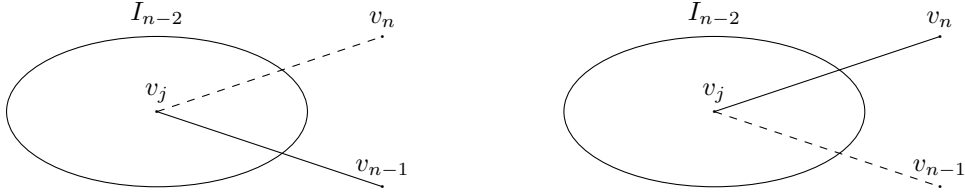
\begin{figure}\centering

\begin{tikzpicture}
\draw (2,2) ellipse (2cm and 1cm) node[above = .4in]{$I_{n-2}$};
\draw (2,2)  node{.} node[above]{$v_j$} -- (5,1) node{.} node[above]{$v_{n-1}$};
\draw [dashed] (2,2)  -- (5,3) node{.} node[above]{$v_{n}$};

\end{tikzpicture}
\hfil
\begin{tikzpicture}
\draw (2,2) ellipse (2cm and 1cm) node[above = .4in]{$I_{n-2}$};
\draw  [dashed] (2,2)  node{.} node[above]{$v_j$} -- (5,1) node{.} node[above]{$v_{n-1}$};
\draw (2,2)  -- (5,3) node{.} node[above]{$v_{n}$};

\end{tikzpicture}

\caption{Reshuffling neighbors.}
\label{figure:swap}

\end{figure}

\subsection{No-Gaps and No-Structure Delocalization} Step 1 of our modified plan has two substeps.
\begin{itemize}
\item  Substep 1(i). We first show that with high probability, $\|\Bv_{I}\|_{2}$ is not too small for all $I$ of size $\Theta(n)$ (Theorem \ref{thm:non-gap:lap}).
\vskip .1in
\item  Substep 1(ii). We then show that $\Bv_{I}$ does not have structure (Theorem \ref{thm:LCD:lap}).
\end{itemize}

The first type of result is called no-gaps delocalization, a notion pioneered by Rudelson and Vershynin in \cite{RV-nongap}. Note that in random matrix theory, delocalization usually means that the eigenvectors are flat, not localized. For instance, for the current centered Laplacian model, it has been shown in \cite{HL-Laplacian} that all unit eigenvectors $\Bv$ corresponding to the eigenvalues in the bulk are completely delocalized in the sense that $\|\Bv\|_{\infty} \le n^{{-1/2+o(1)}}$ with very high probability. So $\|\Bv_{I}\|_{2}$ is small over any set $I$ of size $o(n)$.  The no-gaps delocalization, on the other hand, addresses the property that $\|\Bv_{I}\|_{2}$ cannot be small over any set $I$ of order  $n$. We will cite below a very strong result of  Rudelson-Vershynin in \cite{RV-nongap}, that works for any general random non-symmetric, symmetric and skew-symmetric matrices $X_{n}$ (where $x_{ij}$ and $x_{ji}$ can depend on each other, but otherwise are independent, subgaussian of mean zero and variance one).

\begin{theorem}\label{thm:RV_no_gap} There exist constants $C, C', c$ such that the following holds. Let
    \[
        \la_{n} \geq \frac{1}{n} \text{ and } s \geq C \la_{n}^{-7/6} n^{-1/6} + \exp (-c/ \sqrt{\la_{n}}).
    \]
    Then with probability at least $1 - (C' s)^{\la_{n} n}$, every unit eigenvector $\Bv$ of $X_{n}$
\[
\|\Bv_I\|_{2} \geq (\la_{n} s)^6 
\]
for all $I \subset [n]$ such that $|I| \geq \la_{n} n$. 
\end{theorem}
Although this result is extremely strong, the method of \cite{RV-nongap} does not seem to extend to matrices where the diagonal entries depend on other entries (the way $\wb{L}_{n}$ does), nor to matrices of norm of order $\omega(\sqrt{n})$ --- in our setting, $\wb{L}_{n}$ has norm of order $\Theta( \sqrt{n \log n})$ with high probability.

In this paper, we show the following analogue

\begin{theorem}[Affine no-gaps delocalization of eigenvectors of $\wb{L}'_{n-1}$]\label{thm:non-gap:lap}
Let $n_0 = \lambda_n n$, where 
$$\frac{1} {\log \log n}\le \lambda_n < 1.$$ 
For some sufficiently large constant \( A_0 \), the probability with respect to $\wb{L}'_{n-1}$ (or $\wb{L}_{n-1}$) that it has a unit eigenvector $\Bv$ and an index set $I \subset [n]$ of size $n_0$ such that $\Bv \perp \1$ and
\begin{equation}\label{eqn:segment:approx}
\inf_{a\in \R} \|\Bv_I -a1_I\|_2 \le \frac{1}{ (\log n)^{\frac{A_0}{\lambda_n}}}
\end{equation}
is bounded by $n^{-\omega(1)}$ as $n\to \infty$.  
\end{theorem}
As a corollary, when $\la_n$ is of constant order, with high probability, for any non-trivial unit eigenvector $\Bv$, not only is the total mass over $I$ not small,  $\|\Bv_I \|_2 \ge \frac{1}{ (\log n)^{O(1)}}$, but it is also not small after shifting (affine delocalization):
$$\inf_{a\in \R} \|\Bv_I -a 1_I\|_2 \ge \frac{1}{ (\log n)^{O(1)}}.$$ 

This translates into a statement about the ``incompressibility'' of affine shifts of segments of eigenvectors. It is possible that the result continues to hold for $\inf_{a\in \R} \|\Bv_I -a 1_I\|_2 \ge \Theta_{\la_n}(1)$ and for $\la_n \ll \frac{1} {\log \log n}$ (with possibly much finer approximation than \eqref{eqn:segment:approx}) as in \cite{RV-nongap}, but it seems that $ \frac{1}{ (\log n)^{O(1)}}$ (for the mass contribution) and $\frac{1} {\log \log n}$ (for the lower bound of $\la_n$) are the limits of our proof. Theorem \ref{thm:non-gap:lap} applies to subsets $I$ of size larger than $n/\log \log n$ which is enough for our purposes, and also, the affine aspect is new and will play a key role for later parts of the paper. While the result is much weaker compared to Theorem \ref{thm:RV_no_gap}, our matrix model, the random Laplacian, is more involved and we develop new ways to handle these difficulties. Furthermore, the behavior of eigenvectors of the Laplacian is fundamentally different than that of the non-symmetric or Wigner matrices. Even for a Laplacian matrix generated from Gaussian random variables, the eigenvectors are not uniformly distributed over the sphere.  In fact, simulations show that eigenvectors corresponding to large eigenvalues appear to be mildly localized.

Next, we will discuss Substep 1(ii), which is another innovative aspect of our paper. To do this, we first introduce an important concept following Rudelson and Vershynin.

\begin{definition}\label{def:LCD}
   For a unit vector \( \Bx\in  \R^{n} \), we define the \textbf{least common denominator} of \( \Bx \) with parameters \( \kappa > 0 \) and \( \gamma \in (0,1) \) as
   \begin{equation*}
       \LCD_{\kappa, \gamma} (\Bx) = \inf \left\{ \theta > 0 : \dist (\theta \Bx, \Z^{n}) < \min\{ \gamma \|\theta \Bx\|_{2}, \kappa\}\right\}.
   \end{equation*}
For notational convenience, whenever we refer to $\LCD_{\kappa, \gamma} ({\Bx})$ for any non-zero vector $\Bx \in \R^{n}$, we always mean $\LCD_{\kappa, \gamma} \left(\frac{\Bx}{\|\Bx\|_{2}}\right)$.
\end{definition}

In this paper, if not specified otherwise, we will choose $\kappa = \kappa_{n}=n^{c}$ for some sufficiently small constant $c$, while $\gamma=\gamma_{n}$ is sufficiently small, which can tend to zero slowly. 

One of the highlights of the paper is to show that, not only are the eigenvectors of the Laplacian no-gaps delocalized, but also ``no-structure delocalized''.

\begin{theorem}[No-structure delocalization of Laplacian eigenvectors]\label{thm:LCD:lap} Let $A>0$ be given. Let $n_0 = \lambda_n n$, where
$$\frac{1} {\sqrt{\log \log n}}\le \lambda_n < 1.$$ 
The probability with respect to $\wb{L}'_{n-1}$ (or $\wb{L}_{n-1}$) that it has a unit eigenvector $\Bv$ and an index set $I \subset [n]$ of size $n_0$ such that $\Bv \perp \1$ and
$$\LCD_{\kappa,\gamma}(\Bv_{I}) \ge n^{A}$$
is bounded by $n^{-\omega(1)}$ as $n\to \infty$.
\end{theorem}

Here the bound $\frac{1} {\sqrt{\log \log n}}\le \lambda_n$ can likely be improved to $\frac{1} {\log \log n}\le \lambda_n$ as in Theorem \ref{thm:non-gap:lap}, but we will not dwell on this point here. 

\subsection{Approximate Eigenvectors and More General Results}\label{sub:approx} Given $A$ from Theorems \ref{thm:non-gap:lap} and \ref{thm:LCD:lap}, we choose $C$ to be a sufficiently large constant. By approximating the eigenvalues $\la$ with scale $n^{-C}$ and using Lemma \ref{lemma:norm:lap}, it suffices to consider the event that there exists $\Bv \in S^{n-2}$ and $\la_{0} \in \CN_{\la} =\{\frac{k}{n^{C}} , k\in \Z, |k| \le n^{C+1}\}$ such that $|(\wb{L}'_{n-1}- \la_0) \Bv|\le n^{-C}$, and $v_1+\dots+ v_n=0$ \footnote{As the size of $\CN_{\la}$ is of order $|\CN_{\la}| =n^{O(1)}$, while our probability is of order $n^{-\omega(1)}$, it suffices to fix one $\la_{0}$ from $\CN_{\la}$.}. Let 
\begin{equation}
D :=\lceil n^{C+1/2} \rceil.
\end{equation} 
By further $\sqrt{n}/D$-approximating $\Bv$ in $l_2$-norm, we can assume that the entries of $\Bv$ are in $\frac{1}{D}\Z$, and 
\begin{equation}\label{eqn:Lv}
\|(\wb{L}'_{n-1}- \la_0) \Bv\|_2 =O(  \sqrt{n} \log n \times n^{-C}); \mbox{ and } |\sum_i v_i| \le n/D =n^{1/2-C}; \mbox{ and } \|\Bv\|_2 = 1+ O(n^{-C}).
\end{equation}
We call the above $\Bv$ an {\it approximate non-trivial eigenvector}.

\begin{theorem}\label{thm:approx:eigenvector} Theorems \ref{thm:non-gap:lap} and \ref{thm:LCD:lap} hold for the approximate eigenvectors from \eqref{eqn:Lv} of $\wb{L}'_{n-1}$ (or of $\wb{L}_{n-1}$).
\end{theorem}

\subsection{Overcrowding estimate}\label{sub:overcrowding:discussion} One of the main obstacles in ruling out the event that $\|\Bv_{I}\|_{2}=o(1)$ or the event that $\LCD(\Bv_{I})$ is small is that we have no information on the remaining vector $\Bv_{I^{c}}$. For the sake of discussion, let us assume that $X_{n} \Bv =0$. We can rewrite this as $M_{11}\Bv_{I} + M_{21} \Bv_{I^{c}} =0$ and $M_{12}\Bv_{I} + M_{22} \Bv_{I^{c}} =0$, where 
$$M_{11}= ({X_{n}})_{I \times I}, M_{12}=  ({X_{n}})_{I \times I^{c}}, M_{21}=  ({X_{n}})_{I^{c} \times I}, M_{22} =  ({X_{n}})_{I^{c} \times I^{c}}, \mbox{ and } X_{n} =
\begin{bmatrix}
\multicolumn{3}{c|}{M_{11}} & & M_{12} \\
\hline
\multicolumn{3}{c|}{} & & \\
\multicolumn{3}{c|}{M_{21}} & & M_{22} \\
\end{bmatrix}.$$
Given $\Bv_{I}$ from some finite set, for which we can control the size, conditioning on $M_{21}, M_{12}$ and $M_{22}$ we can solve for $\Bv_{I^{c}}$ from $M_{21}\Bv_{I} + M_{22} \Bv_{I^{c}} =0$, that $\Bv_{I^{c}} = - M_{22}^{{-1}}M_{21}\Bv_{I}$. Substituting this into the first equation, we hence obtain
$$(M_{11}- M_{12} M_{22}^{{-1}}M_{21})\Bv_{I} =0.$$
We can then bound this probability using the randomness on the off-diagonal entries of $M_{11}$. 

One of the key problems of this approach, taking into account that we are working with asymptotic approximations rather than with identities, is that we have to control the norm of $\|M_{22}^{-1}\|_{2}$. This seems to be a very hard problem, where it is only known recently from \cite{NgW2} that $M_{22}$ is invertible with high probability. To avoid this difficulty, we will not work with $M_{22}$ directly, but a near-square rectangular submatrix $M_{22}'$ of it, where we can control the non-zero least singular values quite effectively. One drawback of this modification is that we can no longer solve for $\Bv_{I^{c}}$, but for only (and approximately) a major part of it. However, this problem can be resolved by again passing to another approximate eigenvector. We now state the main result on the strong invertibility of $M_{22}'$ in the form of $L_{n}$.

\begin{theorem}[Overcrowding for spectrum of Laplacian matrices]\label{thm:lap_overcrowding}\label{lemma:square:lap} Let $L_{n}$ be the Laplacian corresponding to an Erd\H{o}s-R\'enyi graph $G(n,p)$. There exist constants $C \geq 1$,  $c > 0$, depending only on $p$ such that 
\begin{equation*}\label{eq:lap_overcrowding}
        \mathbb{P}\left( \sigma_{n-k+1}(L_{n}+ F) \leq \frac{c k}{\sqrt{n}}\right)  =O\Big(\exp(-\Theta(k^{3/2}/\log n)\Big),
\end{equation*}
        for $k \geq C \log n$ and where the implied constants depend on $C,c$.
\end{theorem}
This result, which can be seen as an analogue of Theorem 1.7 in \cite{NgJFA} where i.i.d.~and symmetric matrices were considered, is interesting in its own right. Interestingly, as we can see, the result works for all choices of deterministic matrices $F$, and hence works for $\wb{L}_{n}$ as well \footnote{The proof automatically extends to the model $\wb{L}'_{n}$.}.

The dependence on $k$ in the exponential probability bound in Theorem \ref{thm:lap_overcrowding} can likely be improved to $k^2$ as in \cite{NgJFA}. However, we do not pursue this extension further in this paper since, for our application, it suffices that when $k = \Theta(n/\log^C n)$, the probability bound is super exponentially small.

\subsection{Small coordinates of eigenvectors}
  
As first observed in \cite{NTV}, a direct consequence of controlling the structure of the eigenvectors of a principal minor is that the eigenvectors of the Laplacian cannot have many coordinates of small size (in fact, we show they cannot have more than one zero coordinate).  The zero coordinates of Laplacian eigenvectors play a special role in nodal domains in spectral graph theory and the physics of mechanical systems, where they are referred to as soft nodes \cite{DLL-nodal, BHL-nodal, CKK-soft, caputo2023eigenvectors}.  For dynamical systems on graphs governed by the Laplacian, soft nodes are of interest theoretically and practically because they are immune to forcing and damping \cite{caputo2012oscillations}.

In the random matrix setting, the following result can be seen as another instance of no-gaps delocalization in the extreme case where the subset is of constant size.  
\begin{theorem} [Small coordinates of eigenvectors] \label{thm:smallcoordinates}
	For any $A$, there exists a $B$, depending on $A$, such that with probability at least $1 - n^{-A}$, an eigenvector of $L_n$ or $\wb{L}_n$ does not have more than one coordinate of size less than $n^{-B}$.
\end{theorem} 

Naturally, we expect that there are no small coordinates with high probability, as was shown for Wigner matrices in \cite{NTV}.  

\subsection{Further Remarks} By following our method, the no-structure delocalization result, Theorem \ref{thm:LCD:lap}, should hold for other models such as non-symmetric, symmetric, and skew symmetric matrices of independent subgaussian entries of mean zero and variance one. In fact, we suspect that for these models, the result should be stronger, that $\la_{n}$ can be as small as $n^{-c}$ for some positive constant $c$ (or even as small as $n^{-1+c}$), and the probability bound $n^{-\omega(1)}$ might be improved to a subexponential rate at least. 

Additionally, it seems to be an important problem to extend these results to Laplacians of sparse graphs $G(n,p)$, where $p \rightarrow 0$ as $n \rightarrow \infty$, as well as to random $d$-regular graphs (for either fixed $d$ or $d\to \infty$ with $n$). Finally, a natural direction for future research is to study the law of the minimum gap of these random Laplacians, with a focus on demonstrating universality.

\section{Supporting Lemmas}\label{Section:supporting}

We will frequently make use of the following deterministic lemma.  
\begin{lemma}[Theorem 6 of \cite{NY}] \label{lem:lowerbound}
        Let \(A:\R^{n} \to \R^{k}\) of full rank \(k\). Then for every \(1 \leq l \leq k - 1\), there exists an \(I \subseteq  \{1, \ldots, n\} \) of size $l$ such that if \(A_I\) denotes \(A \mid_{\R^{I}}\), then
        \begin{align*}
            \sigma_l(A_{I}) \geq c \max_{r \in \{l + 1, \ldots, k\}}\sqrt{\frac{(r - l)\sum_{i=r}^{n} \sigma_i(A)^2}{nr}},
        \end{align*}
        for an absolute constant $c$.
    \end{lemma}
Next, we need the following elementary fact for the centered Laplacian matrices.
\begin{lemma}[Norm of centered Laplacian matrix]\label{lemma:norm:lap} We have 
$$\P(\|\wb{L}_{n}\|_2 \ge \la \sqrt{n} \log n) \ll \exp(-\Theta(\la^2 \log n)).$$
In particular, with probability at least $1- \exp(-\Theta(\log^2n))$, all submatrices of $\wb{L}_{n}$ have norm bounded by $\sqrt{n} \log n$.
\end{lemma}
\begin{proof}[Proof of Lemma \ref{lemma:norm:lap}]
    By the triangle inequality 
    $$\|\wb{L}_{n}\|_{2} \le  \|D_{n}-\E D_{n}\|_{2} + \|A_{n} - \E A_{n}\|_{2}.$$
    We can then bound $\|D_{n}-\E D_{n}\|_{2}$ and $\|A_{n} - \E A_{n}\|_{2}$ separately. Alternatively, it follows from \cite[Theorem 3.2]{Band} that the median of $\|\wb{L}_{n}\|_2$ has order $\Theta(\sqrt{ n \log n})$. Then by Talagrand concentration theorem, as $\|\wb{L}_{n}\|_2$ is convex and $\sqrt{n}$-Lipschitz with respect to the $\|.\|_{HS}$-norm, we have 
    $$ \P(|\|\wb{L}_{n}\|_2 - M(\|\wb{L}_{n}\|_2)| \ge \sqrt{n} \la ) \le \exp(-\Theta(\la^2)).$$
\end{proof}

Occasionally, in many proofs we will need the following tensorization lemma to transform anti-concentration bounds from independent random variables to random vectors (\cite[Lemma 3.4]{RV-LO}).
\begin{lemma}\label{lem:tensorization}
  Let $X = (x_1, \ldots, x_n)$ be a random vector in $\R^n$ with independent 
  coordinates $X_k$. Assume that $\P(|x_{i}| < \delta) \le K \delta$ for all $\delta \ge \delta_{0}$. Then
  $$ \P(x_{1}^{2}+\dots + x_{n}^{2} < \delta^{2} n) \le (C_{0}K \delta)^{n} $$ 
  for some absolute constant $C_{0}$.    
\end{lemma}

\begin{proof}[Proof of Lemma \ref{lem:tensorization}]
Let $\delta \ge \delta_0$, we have the following by Chebyshev's inequality:
$$
\P \Big( \sum_{i=1}^n x_i^2 < \delta^2 n \Big)
  = \P \Big( n - \frac{1}{\delta^2} \sum_{i=1}^n x_i^2 > 0 \Big)
  \le \mathbb{E} \exp \Big( n - \frac{1}{\delta^2} \sum_{i=1}^n x_i^2 \Big) 
  = e^n \prod_{i=1}^n \mathbb{E} \exp(-x_i^2/\delta^2).
$$
Integrating the tail bound, we have 
$$
\mathbb{E} \exp(-x_i^2/\delta^2) = \int_0^1 \P \big( \exp(-x_i^2/\delta^2) > s
\big) \; ds = \int_0^\infty 2 u e^{-u^2} \, \P(|x_i| < \delta u) \; du.
$$
Decomposing the integration and using the assumption that $\P(|x_i| < \delta u) \leq \P(|x_i| < \delta) \leq K \delta$ for $u \in [0,1]$, we have $$\mathbb{E} \exp(-x_i^2/\delta^2) \leq \int_0^1 2 u e^{-u^2} K \delta \; du + \int_1^\infty 2 u e^{-u^2} K \delta u \; du \leq C_0 K \delta.
$$ 
Plugging this back to the first inequality, we have $$\P \Big( \sum_{i=1}^n x_i^2< \delta^2 n \Big) <e^n (C_0 K \delta)^n.$$ 
\end{proof}

For the remainder of this section we will discuss several properties of vectors having small $\LCD$, including compressible and incompressible vectors. First, the following  vectors are outputs of our no-gaps delocalization result Theorem \ref{thm:non-gap:lap}.

\begin{claim}\label{claim:non-gap:spread} Let $0<c_{1} \le  c_{0}<1< C$ be given parameters (that might depend on $n$). Assume that $\Bx\in S^{n-1}$ be such that $\|\Bx_{I}\|_{2} \ge c_{1}$ for any $I \subset [n]$ of size at least $\lfloor c_{0}n \rfloor$. Then the following holds.
\begin{itemize}
\item For any set $S$ of size $\lfloor c_{0}n \rfloor$, there exists an $i\in S$ such that $|x_{i}| \ge (c_{1}/c_{0})^{{1/2}}/\sqrt{n}$. It thus follows that for all but $\lfloor c_{0}n \rfloor -1$ indices $i$ we have $|x_{i}| \ge (c_{1}/c_{0})^{{1/2}}/\sqrt{n}$.
\vskip .1in
\item The set $S_{c_{0},c_{1},C}$ of indices $i$ where 
$$(c_{1}/c_{0})^{{1/2}}/\sqrt{n} \le |x_{i}| \le C/\sqrt{n}$$ 
has size at least $(1- c_{0}-C^{-2})n$.
\end{itemize}
\end{claim}
\begin{proof} It suffices to show the first part. Assume otherwise, then we would have $\|\Bx_{S}\|^{2}\le |S| (c_{1}/c_{0})/n <c_{1}$, contradiction.
\end{proof}

Given parameters $c_{0},c_{1},C$, we define the \( \mathrm{spread}_{c_{0}, c_{1},C}(\Bx) \) to be the set of  coordinates \( k \) satisfying
   \begin{equation*}
   (c_{1}/c_{0})^{{1/2}}/\sqrt{n} \le |x_{i}| \le C/\sqrt{n}.
   \end{equation*}

The next lemma connects the notions of compressible vectors and \( \LCD \).

\begin{lemma}\label{lem:incomp_lcd} Assume that $\Bx$ satisfies Claim \ref{claim:non-gap:spread}.  Then
\begin{equation*}
        \LCD_{\kappa, \gamma}(\Bx) \geq  \frac{\sqrt{n}}{2C}
    \end{equation*}
    for any $\kappa >0$ and $\gamma \le \frac{1}{2}[(1- c_{0}-C^{-2}) (c_{1}/c_{0})]^{{1/2}}$.
\end{lemma}
It is worth noting that as \( \gamma \) decreases, $\LCD_{\kappa,\gamma}$ increases.
\begin{proof}[Proof of Lemma \ref{lem:incomp_lcd}]
    For each \( k\in \mathrm{spread}(\Bx) \),
    \begin{equation*}
   (c_{1}/c_{0})^{{1/2}}/\sqrt{n} \le |x_{k}| \le C/\sqrt{n}.
    \end{equation*}
   For all \( \theta \in (0, \frac{\sqrt{n}}{2C}) \), and \( k\in \mathrm{spread}(\Bx) \),
    \begin{equation*}
        \mathrm{dist}(\theta x_k, \Z) =|\theta x_{k}| \geq  \theta  (c_{1}/c_{0})^{{1/2}}/\sqrt{n}.
    \end{equation*}
    Therefore, for \(  \theta \in (0, \frac{\sqrt{n}}{2C})  \),
    \[
        \mathrm{dist}(\theta \Bx, \Z^{n}) \geq \mathrm{dist}(\theta \Bx|_{\mathrm{spread}(\Bx)}, \Z^{\mathrm{spread}(\Bx)}) \geq (1- c_{0}-C^{-2})^{1/2}   (c_{1}/c_{0})^{{1/2}} \theta > \gamma \theta.
    \]
    Therefore, 
    \begin{equation*}
        \LCD_{\kappa, \gamma}(\Bx) \geq \frac{\sqrt{n}}{2C}.
    \end{equation*}
\end{proof}

\begin{fact}[$\LCD$ of subvectors]\label{LCD:scaling} Assume that $\Bx'$ is a subvector of $\Bx$ such that $\gamma \|\Bx\|_{2}/\|\Bx'\|_{2} <1$. We have 
\[\LCD_{\kappa, \gamma \|\Bx\|_{2}/\|\Bx'\|_{2}}\left(\frac{\Bx'}{\|\Bx'\|_{2}}\right) \le \frac{\|\Bx'\|_{2}}{\|\Bx\|_{2}} \LCD_{\kappa, \gamma}\left( \frac{\Bx}{\|\Bx\|_{2}}\right).\]

\end{fact}

\begin{proof}[Proof of Fact \ref{LCD:scaling}] It suffices to assume $\|\Bx\|_{2}=1$. Note that for any $D>0$, if $\dist (D \Bx, \Z^{n}) \le \min(\gamma D, \kappa)$ then trivially 
$$\dist((D\|\Bx'\|_{2})(\Bx'/ \|\Bx'\|_{2}), \Z^{n}) \le \min((\gamma/\|\Bx'_{2}\|_{2}) \| (D\|\Bx'\|_{2})(\Bx'/ \|\Bx'\|_{2})\|_{2},\kappa).$$ 
\end{proof}
One of the main features of $\LCD$ is that it detects the level of radius where we can bound the small ball probability asymptotically optimally (up to a multiplicative constant).

\begin{theorem}[Small ball probability via LCD, \cite{RV-LO}]\label{thm:small_ball_prob_LCD}
    Let \( \xi \) be a subgaussian random variable of mean zero and variance one, and let \( \xi_{1}, \cdots, \xi_{n} \) be i.i.d.~copies of \( \xi \). Consider a vector \( \Bx\in \R^{n} \) which satisfies \( \left\|\Bx\right\| \geq 1 \). Then, for every \( \kappa > 0 \) and \( \gamma \in (0,1) \), and for
    \begin{equation*}
        \varepsilon \geq \frac{1}{\LCD_{\kappa, \gamma}(\Bx)},
    \end{equation*}
    we have
    \begin{equation*}
    \rho_{\varepsilon}(\Bx) := \mathcal{L}(\sum_{i=1}^{n} x_i \xi_i, \varepsilon) := \sup_{a\in\R}\mathbb{P} \left(\left| \sum_{i=1}^{n} x_i \xi_i - a\right| \leq \varepsilon\right) = O \left(\frac{\varepsilon}{\gamma} + e^{- \Theta(\kappa^2)}\right),
    \end{equation*}
    where the implied constants depend on \( \xi \).
\end{theorem}

Finally, we need the following important results, the proofs of which will be given for the reader's convenience. 

\begin{lemma}[Nets of structured vectors]\label{lemma:nets} Let $\gamma, \al, \kappa, D_{0}$ be given positive parameters, where $\gamma<1$, and $\kappa$ and $D_{0}$ are sufficiently large. The following holds for $m$ sufficiently large.
 \begin{itemize} 
 \item (Nets of vectors having $\LCD$ belonging to $(D_{0},2D_{0}]$, \cite[Lemma 4.7]{RV-rectangular}\cite[Lemma 11.3]{NTV}) Assume that $\kappa/D_{0} \le 1/2$. Let $S_{\al, D_{0}}$ be the set of $\Bx\in \R^{m}$ with $\|\Bx\| =\al$ and that 
$$\LCD_{\kappa,\gamma}(\Bx/\|\Bx\|_{2}) \in [D_{0}, 2D_{0}].$$ 
Then there exists a $(2 \al \kappa/D_0)$-net of $S_{\alpha, D_0}$ of cardinality at most $(C_0  D_0/\sqrt{m})^m$, where $C_0$ is an absolute constant. Consequently, there exists a $(2\kappa/D_0)$-net of $S_{\al,D_0}$ of cardinality at most $(C_0  (\al+1) D_0/\sqrt{m})^m$. 
\vskip .1in
\item (Nets of unit vectors having $\LCD$ smaller than $D$, \cite[Lemma 11.4]{NTV}) Assume that $2\kappa \le D_0 \le D$. Then the set $S_{1,D_0}$ has a $(2\kappa/D)$-net of cardinality at most $(C_0D/\sqrt{m})^m$ for some absolute constant  $C_0$.
\vskip .1in
\item (Trivial nets) The number of vectors in $(\frac{1}{D}\Z)^m$ that have norm bounded by $\al$ is at most $(C' \al D/\sqrt{m} +1 )^m$, and hence this set of vectors accepts an $\sqrt{m}/D$-net of size $(C' \al D/\sqrt{m} +1 )^m$.  
\end{itemize}
\end{lemma}

\begin{proof}[Proof of Lemma \ref{lemma:nets}]
Let us first focus on the net of $S_{\al, D_{0}}$. For $\Bx\in S_{\al, D_0}$ and $\|\Bx\|_{2}=\al$, denote
$$D(\Bx):= \LCD_{\kappa,\gamma}(\Bx/\al).$$
By definition, $D_0\le D(\Bx)\le 2D_0$ and there exists $p\in \Z^m$ with 
$$\left\|\frac{\Bx}{\al} -\frac{p}{D(\Bx)}\right\|_{2} \le \frac{\kappa}{D(\Bx)}.$$ 
This implies that $\|p\| \approx D(\Bx)$. More precisely, it implies
\begin{equation}\label{eqn:LCD:net:1}
1- \frac{\kappa}{D(\Bx)} \le \left\|\frac{p}{D(\Bx)}\right\|_{2} \le  1+ \frac{\kappa}{D(\Bx)}.
\end{equation}
This implies that 
\begin{equation}\label{eqn:LCD:net:2}
\|p\|_{2} \le 3 D(\Bx)/2 \le 3 D_0.
\end{equation}
It also follows from \eqref{eqn:LCD:net:1} that
\begin{equation}\label{eqn:LCD:net:3}
\left\|\Bx -\al \frac{p}{\|p\|}\right\|_{2} \le \al \left\|\frac{\Bx}{\al} - \frac{p}{D(\Bx)}\right\|_{2} + \al\left\| \frac{p}{\|p\|}(\frac{\|p\|}{D(\Bx)} -1)\right\|_{2} \le \frac{2\al \kappa}{D(\Bx)} \le \frac{2 \al \kappa}{D_0}.
\end{equation}
Now set 
$$\CN_0:=\left\{\al \frac{p}{\|p\|}, p\in \Z^m \cap B(0,3D_0)\right\}.$$
By \eqref{eqn:LCD:net:2} and \eqref{eqn:LCD:net:3}, $\CN_0$ is a $\frac{2 \al \kappa}{D_0}$-net for $S_{D_0}$. On the other hand, it is known that the size of $\CN_0$ is bounded by $(C_0\frac{ D_0}{\sqrt{m}})^m$ for some absolute constant  $C_0$.

For the second part of the first statement, it suffices to assume that $\al\ge 1$. As we can cover $S_{\al, D_{0}}$ by $(C_0D_0/\sqrt{m})^m$ balls of radius $2 \al \kappa/D_0$, we can then cover these balls by smaller balls of radius $2\kappa/D$, the number of such small balls is at most $(O(\al))^m$. Thus there are at most $(C_0\frac{ \al D_0}{\sqrt{m}})^m$ balls  of radius $2\kappa/D$ in total.

We can justify the second statement similarly. By the first part, one can cover $S_{1,D_0}$ by $(C_0D_0/\sqrt{m})^m$ balls of radius $2\kappa/D_0$. We then cover these balls by smaller balls of radius $2\kappa/D$; the number of such small balls is at most $(O(D/D_0))^m$. Thus, there are at most $(O(D/\sqrt{m}))^m$ balls in total.
\end{proof}

\section{Overcrowding Estimate: Proof of Theorem \ref{thm:lap_overcrowding}}\label{sect:overcrowding} 

We first introduce an important lemma (whose proof is delayed to the end of the section) to control the distance of a random vector with non-i.i.d.~entries to a subspace.
\begin{lemma}\label{lem:distance}
    Let $\Bv \in \R^n$ be a random vector whose entries are independent (not necessarily identically distributed) and $\Bv_i$ have mean zero, $\var(\Bv_i) \geq 1$ and $|\Bv_i| \leq T$ for some parameter $T$ with probability one.  Then, if $P_{ H^\perp}$ is a deterministic orthogonal projection in $\R^n$ onto a subspace of dimension $k$, there exist constants $C, c, c'$ depending on $p$ such that for any $0 < t < 1/2,$
    \[
    \sup_{\Bu \in \R^n}\P(\|P_{H^\perp} \Bv - \Bu\|_2 \leq t T \sqrt{k} - c' T) \leq C \exp(-c t^2k)
    \]
\end{lemma}

\begin{proof}[Proof of Theorem \ref{thm:lap_overcrowding}]
    We follow the proof strategy in \cite{NgJFA}. Observe that as $L_n+F$ is symmetric, $\sigma_{n-k+1}(L_n+F) \leq \frac{c k}{\sqrt{n}}$ is equivalent to the event that there exists an $i$ such that
    \[
    \frac{-ck}{\sqrt{n}} \leq \lambda_i \leq \lambda_{i-k+1} \leq \frac{c k}{\sqrt{n}}.
    \]
    We let $I := [-ck/\sqrt{n}, ck/\sqrt{n}]$.  Let us assume that $\lambda_i \in I$ and $\Bv_i$ is such that $(L_n+F) \Bv_i = \lambda_i \Bv_i$.  As $L_n+F$ is symmetric, the eigenvectors $v_i$ are orthogonal, so that
    \begin{equation} \label{eq:smallnorm}
        \left\|(L_n+F) \Bv_i \right\|_2 \leq \frac{ck}{\sqrt{n}}.
    \end{equation}

    Write $\Bv_j = (v_{j1}, \ldots, v_{jn})^{T}$, and let \(\Bc_i\) denote the \(i^{\text{th}}\) column of \(L_n+F\). Let  \(V\) be the  \(k\times n\) matrix formed by the row vectors \(\Bv_{i}^{T}\) and let $\Bw_1, \dots, \Bw_n$ be its columns.  Clearly, \eqref{eq:smallnorm} is equivalent to
    \[
    \left\| \sum_{j=1}^n v_{ij} \Bc_j \right\|_2 \leq \frac{ck}{\sqrt{n}} \text{\quad for all $1 \leq i \leq k$}.
    \]
    For $J \subset [n]$, we use $V_J$ to indicate the matrix formed by the columns $\Bw_j$ with $j \in J$. Observe that by construction,
    \[
    \sum_{1 \leq j \leq n} \|\Bw_j\|_2^2 = k.
    \]
    We now show that there exists a well-conditioned minor (submatrix) of $V$.  By Lemma \ref{lem:lowerbound}, for any $1 \leq l \leq k-1$, there exist distinct indices $i_1, \dots, i_l$ such that 
    \begin{equation} \label{eq:well-invert}
    \sigma_l(Z_{i_1, \dots, i_l} )\geq c \max_{r \in \{l+1, \cdots, k\}} \sqrt{\frac{(r-l) \sum_{i=r}^n \sigma_i(V)^2}{nr}} \geq c' \sqrt{\frac{(k-l)^2}{nk}} 
    \end{equation}
    for some constant $c'$, where the final inequality follows from setting $r = \lceil(l+k)/2 \rceil$ and noting that $\sigma_j(V) = 1$ for $j \in [k]$.
    
    For notational convenience, we let \(Y\) be the \(l \times k\) matrix \(Z_{(i_1, \ldots, i_l)}^{T}\) and $X$ be a right inverse of $Y$ so that $YX = I_l$.
    We define
    \begin{equation} \label{eq:partition}
        A = (\Bc_{i_1}, \ldots, \Bc_{i_l})Y + (\Bc_{i_{l+1}}, \ldots, \Bc_{i_n})Y',
    \end{equation}
    where \(Y' = Z_{(i_{l+1}, \ldots, i_n)}^{T}\).  
    Multiplying \eqref{eq:partition} on the right by $X$ yields
    \[
    AX = (\Bc_{i_1}, \ldots, \Bc_{i_l}) + (\Bc_{i_{l+1}}, \ldots, \Bc_{i_n})Y'X.
    \]
    Recall that by assumption, each column of $A$ has norm bounded by $ck/\sqrt{n}$ and by \eqref{eq:well-invert}, 
    \[
    \|X\| \ll \sqrt{\frac{kn}{(k-l)^2}}.
    \]
    Thus,
    \begin{equation} \label{eq:HS}
    \|AX\|_{HS} \leq \|A\|_{HS} \|X\| \ll \frac{k^2}{(k-l)}.
    \end{equation}

    Let \(H\) denote the subspace spanned by \((\Bc_{i_{l+1}}, \ldots, \Bc_{i_n})\). Equation \eqref{eq:HS} implies that
    \begin{align*}
        \dist(\Bc_{i_1}, H)^2 + \ldots + \dist(\Bc_{i_l}, H)^2 \ll \frac{k^4}{(k-l)^2 }.
    \end{align*}
   Thus, we have reduced the event in \eqref{eq:smallnorm} to the event that
   \[
   \dist(\Bc_{i_1}, H) = O\left(\frac{k^2}{(k-l) }\right) \wedge \cdots \wedge \dist(\Bc_{i_l}, H) = O\left(\frac{k^2}{(k-l) } \right),
   \]
   which we denote by $\mathcal{E}_{i_1, \dots, i_l}$. However, due to the symmetry and structure of the Laplacian, these subevents are related in a complicated way. 
   To begin decoupling the events, we make the following observation.  For any $I \subset [n]$,
   \begin{equation} \label{eq:projection}
   \dist(\Bc_i, H) \geq \dist(\Bc_{i,I}, H_I),
   \end{equation}
   where $c_{i,I}$ and $H_{I}$ are the projections of $\Bc_{i}$ and $H$ onto the coordinates indexed by $I$.
   Without loss of generality, we assume that \((i_1, \ldots, i_l) = (1, \ldots, l)\).
   We utilize \eqref{eq:projection} to note that $\mathcal{E}_{1,2, \dots, l}$ implies the event
   \begin{align*}
    \mathcal{F}_{1,\dots,l} &:= \left( \dist(\Bc_{1, \{l,\dots, n\}}, H_{k, \dots, n} = O\left(\frac{k^2}{(k-l) } \right)\right) \wedge \dots \wedge \left(\dist(\Bc_{l, \{l,\dots, n\}}, H_{k, \dots, n} = O\left(\frac{k^2}{(k-l) } \right)\right)
   \end{align*}
    Note that now we have that $\Bc_{i, \{l,\dots, n\}}$ is independent of $\Bc_{j, \{l,\dots, n\}}$
    for $i,j \leq l$ and $i \neq j$.  However, due to the structure of the Laplacian,
    $\Bc_{i, \{l,\dots, n\}}$ is not independent of $H_{l, \dots, n}$.
    
    Now, we set $l = k/2$ and we introduce an integer parameter $\tau$, which will be a sufficiently large constant depending on $c$ and $p$, and divide our indices $1, \dots, l$ into $l/2\tau$ sets,
    \[
    J_1 = \{1, \dots, 2\tau\}, J_2 = \{2\tau + 1, 4 \tau\}, \dots, J_{l/2\tau} = \{l - 2\tau+1, \dots, l\}.
    \]
    As $\Bc_i$ is a column of $L_n+F$,
    \[
    \Bc_{i, \{l,\dots, n\}} = \Bd_{i, \{l,\dots, n\}} + \Bf_{1, \{l,\dots, n\}}.
    \]
    We define two vectors for every $i \in [l/\tau]$,  
    \[
    \Bc_i^* = \sum_{j= (i-1)\tau+1}^{i \tau/2} \Bc_{j} = \sum_{j= (i-1)\tau+1}^{i \tau/2} \Bd_{j} + \Bf_{j} := \Bd_i^* + \Bf_i^*
    \]
    and
    \[
    \Bc_i^{**} = \sum_{j = i \tau/2 + 1}^{i \tau} \Bc_j = \Bd_i^{**} + \Bf_i^{**}.
    \]
    Every coordinate of $\bar{\Bd}_{i,\{l,\dots, n\}} := \Bd_{i,\{l, \dots, n\}}^* + \Bd_{i,\{l, \dots, n\}}^{**} $ is an i.i.d.~binomial random variable. For a random variable $X \sim B(2\tau,p)$, by Chernoff's bound, for any $0 < \delta < 1$, 
    \[
    \P(|X - 2 \tau p| \geq \delta 2 \tau p ) \leq 2 \exp(-2 \delta^2 \tau p/3).
    \]
    If we set $\delta = 1/2$ and $\tau = K \log n/k$ for a sufficiently large constant $K$ then 
    \[
    \P(|X - 2 \tau p| < \delta 2 \tau p ) \geq (1 - k/5 n).
    \]
    
    Let $D_i := \{j \in \{l, \dots, n\}: (\bar{\Bd}_{i,\{l,\dots, n\}})_j \in [(1-\delta) 2 \tau p, (1+\delta) 2 \tau p] \}$.
    As the entries are independent, by Chernoff's bound again,
    \begin{equation} \label{eq:goodindex}
    \P(|D_i| \leq (1-k/5 n)(n-l)) \leq \exp(-kp/10 ).
    \end{equation}
    Now, for $i \in [l/2 \tau]$, we introduce the following auxiliary randomness inspired again by the ``switching'' method.  We define a new random variable $\hat{\Bc}^*_i$ where $(\hat{\Bc}_i^*)_j = (\Bc^*_i)_j$ for $j \notin D_i$, and for $j \in D_i$,
    \[
    (\hat{\Bc}^*_i)_j = \begin{cases}
        & (\hat{\Bc}_i^*)_j \text{ with probability } 1/2 \\
        & (\Bd_i^{**})_j + \Bf_i^* \text{ with probability } 1/2.
    \end{cases}
    \]
    Similarly, we define a new random variable $\hat{\Bc}^{**}_i$ where $(\hat{\Bc}_i^{**})_j = (\Bc^{**}_i)_j$ for $j \notin D_i$, and for $j \in D_i$,
    \[
    (\hat{\Bc}^{**}_i)_j = \begin{cases}
        & (\hat{\Bc}_i^{**})_j \text{ with probability } 1/2 \\
        & (\Bd_i^{*})_j + \Bf_i^{**} \text{ with probability } 1/2.
    \end{cases}
    \]
    We observe that $\hat{\Bc}_{i, \{l,\dots, n\}}^*, H_{l,\dots,n}$ has the same joint distribution as $\Bc_{i, \{l,\dots, n\}}^*, H_{l,\dots,n}$. Furthermore, the key benefit of this construction is that, even upon conditioning on the outcome of $H_{l,\dots, n}$, the random vectors $\Bc_i$, and $\Bc_i^*$, $\hat{\Bc}_i^*$ are sufficiently random with independent entries for $j \in I_i$. 

    Next, we note that by the triangle inequality, $\mathcal{F}_{1, \dots, l}$ implies the event
    \[
    \mathcal{G}_{1,\dots, l} := \left( \dist(\hat{\Bc}^*_{1, \{l,\dots, n\}}, H_{l, \dots, n} )= O\left(\frac{\tau k^2}{(k-l) } \right)\right) \wedge \dots \wedge \left(\dist(\hat{\Bc}^*_{l/\tau, \{l,\dots, n\}}, H_{l, \dots, n}) = O\left(\frac{\tau k^2}{(k-l) } \right)\right).
    \]
    These subevents are now independent upon conditioning on the initial randomness of $L_n$.
    Now we introduce a parameter $\alpha \ll n$ and divide $J_1, \dots, J_{l/2\tau}$ into $l/2 \tau \alpha$ groups of size $\alpha$.  We call an index $i \in [l/2 \tau \alpha]$ \emph{good} if $|D_j| > (1 - c/10)(n-l)$ for at least one $J_j$ for $j \in [(j-1) l/2 \tau \alpha+ 1, jl/2 \tau \alpha]$.    
    We define the event $\mathcal{E}$ to be the event that every index in $[l/2 \tau \alpha]$ is good.  By $\eqref{eq:goodindex}$ and independence, the probability that $i \in [l/2 \tau \alpha]$ is not good is at most $\exp(-kp \alpha/10)$ so by a union bound, 
    \begin{equation} \label{eq:Eprob}
    \P(\mathcal{E}^c) \leq (l/2 \tau \alpha) \exp(-kp \alpha/10) \leq \exp(-k p  \alpha/20).
    \end{equation}

    On the event $\mathcal{E}$, we have at least $l/2 \tau \alpha$ good indices.  For any good index $i$, by \eqref{eq:projection}:
    \begin{align*}
    \P\left(\dist(\hat{\Bc}^*_{i, D_i}, H_{l, \dots, n} ) \leq O\left(\frac{\tau k^2}{(k-l) } \right) \right) &\leq \P\left(\dist(\hat{\Bc}^*_{i, D_{i}}, H_{D_{i}} )= O\left(\frac{\tau k^2}{(k-l) } \right) \right),
    &\leq C \exp(-\Theta(k))
    \end{align*}
    where the last inequality follows from Lemma \ref{lem:distance} and $k$ sufficiently large.  Note that our choice of $l$ and the size of $|D_i|$ guarantee that the corank of $H_{D_i}$ is at least $k/3$.  Therefore, by independence of the randomness on the good indices, we have that
    \begin{align*}
    \P(\mathcal{E}_{1, \dots, l}) &\leq \P(\mathcal{E}_{1, \dots, l}|\mathcal{E}) + \P(\mathcal{E}^c) \\
    &\leq (C \exp(-\Theta(k)))^{l/2\tau \alpha} +  \exp(-k p \alpha/20) \leq C \exp(-\Theta(k^2/ \alpha \log(n/k))) +  \exp(-k \alpha/20).
    \end{align*}
	Finally, we set $\alpha = \sqrt{k/\log n}$.
\end{proof}

It remains to prove the distance result. 

\begin{proof}[Proof of Lemma \ref{lem:distance}]
    We show that $f(\Bx) = \|P_{H^\perp} \Bx - \Bu\|_2$ is convex and 1-Lipschitz.  For convexity, we observe that for $0 \leq s \leq 1$, 
    \begin{align*}
    \|P_{H^\perp}  \Bx - \Bu\|_2 &= \|P_{H^\perp}  s \Bx - s \Bu + P_{H^\perp} (1-s) \Bx - (1-s) \Bu\|_2 \\
    &\leq s \|P_{H^\perp}   \Bx -  \Bu\| + (1-s)\| P_{H^\perp} \Bx - \Bu\|_2.
    \end{align*}
    Next, note that for $\Bx, \By \in \R^n$,
    \begin{align*}
        \Big|\|P_{H^\perp}  \Bx - \Bu\|_2 - \|P_{H^\perp}  \By - \Bu\|_2 \Big| &\leq \|P_{H^\perp}  (\Bx - \By)\|_2 \\
        &\leq \| \Bx - \By\|_2,
    \end{align*}
    so that $f(\Bx)$ is 1-Lipschitz.
    Thus, by Talagrand's inequality, for absolute constants $K, \kappa > 0$ and any $\lambda > 0$,
    \begin{equation} \label{eq:talagrand}
    \P(|f(\Bx) - M(f(\Bx)| \geq \lambda T) \leq K \exp(-\kappa \lambda^2),
    \end{equation}
    where $M(f(\Bx))$ is a median of $f(\Bx)$.
    We use a second moment argument to control the median.  We first calculate that
    \[
    \E \|P_{H^\perp}  \Bx -\Bu\|_2^2  = \E \Bx^T P_{H^\perp} \Bx - 2 \E \Bx^T P_{H^\perp} \Bu + \|\Bu\|^2  = \sum_{i=1}^n \E \Bx_i^2 p_{ii} + \|\Bu\|_2^2
    \]
    where $p_{ij} := (P_{H^{\perp}})_{ij}$.
    \begin{align*}
    Y &= \|P_{H^\perp}  \Bx -\Bu\|_2^2 - \E \|P_{H^\perp}  \Bx -\Bu\|_2^2 \\
    &= \Bx^T P_{H^\perp} \Bx - 2  \Bx^T P_{H^\perp} \Bu + \|\Bu\|^2 -  \sum_{i=1}^n p_{ii} \E x_i^2 - \|\Bu\|_2^2 \\
    &= \sum_{ij} p_{ij} (x_i x_j - \delta_{ij}) -  2 \sum_{i=1}^n x_i (P_{H^\perp} \Bu)_i.
    \end{align*}
    Therefore,
    \begin{align*}
    \E Y^2 &= \E \left( \sum_{ij} p_{ij} (x_i x_j - \delta_{ij}) \right)^2 - 4 \E \left(\sum_{i=1}^n x_i (P_{H^\perp} \Bu)_i \right) \left( \sum_{ij} p_{ij} (x_i x_j - \delta_{ij}) \right) + 4 \E \left(\sum_{i=1}^n x_i (P_{H^\perp} \Bu)_{i} \right)^2 \\
     &\leq O\left( T^4 \sum_{ij} p_{ij}^2  + T^4 \sqrt{\left(\sum_i (P_{H^\perp}  \Bu)_i^2 \right) \left(\sum_{ij} p_{ij}^2  \right)}  \right) \\
    &= \left( T^4 k + T^2 \|\Bu\|_2^2 k \right),
    \end{align*}
    where we have invoked the fact that $\sum_{ij} p_{ij}^2 = \sum_{i} p_{ii} = k$. 
    
     Since we make no assumptions on $\|\Bu\|_2$, we need to divide into two cases.

     {\bf Case 1:} $(1 - 2t) T \sqrt{\sum_{i=1}^n \E \Bx_i^2 p_{ii}} \leq \|\Bu\|_2 \leq (1 + 2t) T \sqrt{\sum_{i=1}^n \E \Bx_i^2 p_{ii}}$. In this case, we have that
     \[
    \E Y^2 = O(T^4 k).
     \]
     By Markov's inequality, the median of $|Y|$ is $O(T^2 \sqrt{k})$, which implies that the median of $\|P_{H^\perp}  \Bx -\Bu\|_2^2$ is at least $\sum_{i=1}^n \E \Bx_i^2 p_{ii} + \|u\|_2^2 - O(T^2 \sqrt{k}) \geq T^2 k + \|\Bu\|_2^2 - O(T^2 \sqrt{k})$.  We can therefore deduce that the median of $\|P_{H^\perp}  \Bx -\Bu\|_2$ is at least $\sqrt{T^2 k + \|\Bu\|_2^2  - O(T^2 \sqrt{k})} = \sqrt{T^2 k + \|\Bu\|_2^2} - O(T^2)$. Returning to Talagrand's inequality, \eqref{eq:talagrand}, it follows that
     \begin{align*}
         \P(\|P_{H^\perp}  \Bx -\Bu\|_2 \leq t T \sqrt{k} - O(T^2) ) &\leq \P(\|P_{H^\perp}  \Bx -\Bu\|_2 \leq \sqrt{T^2 k + \|\Bu\|_2^2} t T \sqrt{k} - O(T^2) - T t \sqrt{k} ) \\
         &\leq \P(|\|P_{H^\perp}  \Bx -\Bu\|_2 - M(\|P_{H^\perp}  \Bx -\Bu\|_2) | \geq T t \sqrt{k}) \\
         &\leq C \exp(-c't^2 k).
     \end{align*}

    {\bf Case 2.} Now we consider the case where $\|\Bu\|_2 \leq (1 - 2t) T \sum_{i=1}^n \E \Bx_i^2 p_{ii}$ or $\|\Bu\|_2 \geq (1+ 2t) T \sum_{i=1}^n \E \Bx_i^2 p_{ii}$. If $\|P_{H^\perp}  \Bx - \Bu\|_2 \leq t \sum_{i=1}^n \E \Bx_i^2 p_{ii}$, then by the triangle inequality, either $\|P_{H^\perp}  \Bx \|_2 \leq (1 - t) T \sum_{i=1}^n \E \Bx_i^2 p_{ii}$ or $\|P_{H^\perp}  \Bx \|_2 \geq (1 + t) T \sum_{i=1}^n \E \Bx_i^2 p_{ii}$ which reduces to Case 1.    
\end{proof}

\section{No-Gaps Delocalization: Proof of Theorem \ref{thm:non-gap:lap}  for the Approximate Eigenvectors} \label{sec:nogaps}

First, recall the Erd\H{o}s-Littlewood-Offord result from \cite{NgW2}.

\begin{theorem}\label{theorem:LO} Let $\Bw =(w_1,\dots, w_N) \in \R^N$. 
\begin{itemize}
\item Assume that at least $N'$ elements of the $w_i$ satisfy $|w_i| \ge r$. Let $X=(x_1,\dots, x_N)$ where $x_i$ are i.i.d.~copies of a random variable $\xi$ of mean zero, variance one, and bounded $(2+\eps)$-moment. Then
$$\rho_r(\Bw) := \sup_{a\in \R} \P(|X \cdot \Bw -a| < r)=O(\frac{1}{\sqrt{N'}}).$$
Consequently, if \begin{align}
    \text{for every } I \subset [N] \text{ with } |I| \ge N-N' \text{ we have } \|\Bw_I\|_2^2 \ge |I| r^2,
\end{align} then the above holds. 
\vskip .1in
\item Furthermore, if we assume that for all $w \in \R$, the vector $\Bw-w\1$ has at least $N'$ components that have absolute values at least $r$. 
Then for $X=(x_1,\dots, x_N)$ with $x_i, $ for $1\le i\le N$, being i.i.d.~copies of $\xi$ we have the following affine analog
$$\rho_{L,r}(\Bw):= \sup_{a,w\in \R} \P(|X \cdot (\Bw-w\1)- a| < r) = O(\frac{1}{\sqrt{N'}}).$$
As a consequence, for any $R>r$
\begin{align}\label{eqn:ELO}
    \rho_{L,R}(\Bw)= \sup_{a,w\in \R} \P(|X \cdot (\Bw-w\1)- a| < R) =O(\frac{R}{r\sqrt{N'}}).
\end{align} 
Also, if for any $w$
\begin{align}\label{eqn:pig}
    \|(\Bw-w\1)_{I}\|_2^2 \ge |I| r^2 \text{ for any } I \text{ such that } |I| \ge N-N',
\end{align} then the above conclusion holds.
\end{itemize}
Here the implied constants depend on $\xi$.  
\end{theorem}

Note that for Theorem \ref{thm:non-gap:lap} it suffices to assume 
\begin{equation}\label{eqn:lambda}
\frac{1} {\log \log n}\le \lambda_n < \la_{0}, \end{equation}
where $\la_{0}$ is a sufficiently small positive constant \footnote{This is because the smaller $\lambda_{n}$ is, the stronger our statements become. We postulated this condition on $\lambda_{n}$ primarily due to the union bound in Case 1 below. However, further investigation shows that the events considered in that case for larger $\lambda_{n}$ are, in fact, subevents of those corresponding to smaller $\lambda_{n}$.}.

It what follows, we will mainly work with $\wb{L}_{n}$, the proof for $\wb{L}'_{n}$ (or for $\wb{L}'_{n-1}$) is identical.

\subsection{Decomposition of vectors} Let
\begin{equation}\label{eqn:delta_n}
\delta_n := (\log n)^{-\frac{A_0}{\lambda_n}}, \mbox{ where $A_0$ to be chosen sufficiently large.}
\end{equation}

First, we let $I$ be the set of first $ \left\lceil \la_n n\right\rceil$ indices (in the end we will take union bound over $I$). We will bound the probability of the following event $\mathcal{E}$:
there exists a $\Bv \in (\frac{1}{D}\Z)^n$ satisfying \eqref{eqn:Lv} whose first $\lceil \lambda_n n \rceil$ coordinates are similar in size (in the $l_2$-norm). More precisely, for some $a$
$$\|\Bv_I -a \1/\sqrt{n}\|_2 \le \delta_n$$ 
and 
$$\|(\wb{L}_{n}-\la_0)\Bv\|_2 =O(\sqrt{n} \log n \times n^{-C}).$$

Let 
$$k:=k_n=\lfloor \lambda_n n/4\rfloor.$$ 
We partition the index set $[\lceil \lambda_n n \rceil+1,n]$ into subsequences $J_1,\dots,J_\ell$ of consecutive numbers so that $|J_i|=k$ for $i<\ell$ and $k\leq |J_\ell| <2k$. 

We have
$$k\geq \lambda_n n/8 \quad \textrm{and} \quad \ell<8 \la_n^{-1}.$$

We let $t_{0}>1$ be a parameter whose value is dictated by constraints later in the argument. For instance we can choose 
\begin{equation}\label{eqn:t_0}
t_0 = \log^{64} n.
\end{equation}
Also, let 
\begin{equation}\label{eqn:a_n}
\al_n = \log^{-2}n.
\end{equation}

Given $\Bv$,  we will partition $[n]$ into two subsets, $I_m$ (mixed) and $I_s$ (sparse), where $I_s$ is the union of the intervals $J_{i_{0}}=I, J_{{i_{1}}},\dots $, where the \( i_{k} \) above are defined as follows. At each step $j \ge 0$, find a $\Bv_{J_i}$ that has not been selected previously such that there exists $a_i \in \Z/D$ and a subset $J_i' \subset J_i$ of size at least $|J_i|- \al_n \la_n n$ which satisfies
\begin{equation}\label{eqn:dis}
  \|\Bv_{J_i'} -(a_i \1/\sqrt{n})_{J_i'}\|_2 \le \delta_n t_{0}^{j} .  
\end{equation}

Denote that \( i \) as \( i_{j} \). This is possible, up to some maximal index \( i_{0} \). Hence $I_s$ is associated with a maximum index $0\le i_0 \le 8 \la_n^{-1}$ and a partial ordering of $\{1,\dots, \lfloor \la_n^{-1} \rfloor\}$, where the $J_i$ appeared in each step $i$.

We first pause for a few remarks.

\begin{remark}\label{remark:L:sparse:rho} Each $\Bv$ will also gives rise to a parameter $0\le i_{0} \le 8 \la_{n}^{-1}$ such that all of the intervals from $I_{s}$ are $\delta_{n} t_{0}^{{i_{0}}}$-close to some $a_i \1/\sqrt{n}$. Conversely, if $J_i$ does not belong to $I_s$ (i.e., if \( J_{i} \subset I_{m}\)), then by definition, for any $a\in \Z/D$ and any $J_i'\subset J_i$ such that $|J_i'| \ge |J_i| -  \al_n \la_n n$ we have 
$$\|\Bv_{J_i'} -(a \1/\sqrt{n})_{J_i'}\|_2 > \delta_n t_{0}^{i_{0}+1}.$$ 
This implies that for any $a$, there are at least $ \al_n \la_n n$ entries $v_i\in \Bv_{J_i}$ such that
$$|v_i - a/\sqrt{n}| \ge \frac{\delta_n t_{0}^{i_{0}+1}}{\sqrt{|J_i|}} =  \frac{\delta_n t_{0}^{i_{0}+1} \la_n^{-1/2}}{\sqrt{n}}.$$ 
In particular, Condition \ref{eqn:pig} of Theorem \ref{theorem:LO} is satisfied, and so
\begin{equation}\label{eqn:L:rho(J_i)}
\rho_{L,   \frac{\delta_n t_{0}^{i_{0}+1} \la_n^{-1/2}}{\sqrt{n}}}(\Bv_{J_i}) =O\left(\frac{1}{\sqrt{ \al_n \la_n n}}\right).
\end{equation}
\end{remark}
For convenience, set
\begin{align}\label{eqn:r_n}
    r_n :=   \frac{\delta_n t_{0}^{i_{0}+1/2} \la_n^{-1/2}}{\sqrt{n}},
\end{align}
with an extra factor of $t_0^{1/2}$ to be used later.

\begin{remark}\label{remark:L:sparse:trivialnet} Note that for each $i \le i_{0}$,
$$ t_{0}^i \delta_n \le t_{0}^{i_0} \delta_n < r_n \sqrt{|J_i|}.$$
So, for each fixed $a$, the set $\{\Bv_{J_i} \in \R^{|J_i|}, \|\Bv_{J_i} - a \1/\sqrt{n}\|_2 \le t_{0}^i \delta_n\}$ trivially accepts an $r_n \sqrt{J_i}$-net of size 1.

As a consequence, the set $\{\Bv_{J_i} \in \R^{|J_i|}, \|\Bv_{J_i} - a \1/\sqrt{n}\|_2 \le t_{0}^i \delta_n, \mbox{ for some } a \in \Z/D, |a| \le \sqrt{n}\}$ trivially accepts an $r_n \sqrt{J_i}$-net of size $O(D \sqrt{n})$.
\end{remark}

We let $I_m=[n]\setminus I_s$.
We write $\Bv_m:=\Bv_{I_m}$ and $\Bv_s:=\Bv_{I_s}$. Note that $\Bv_s$ can be $r_{n}\sqrt{|I_s|}$-approximated by a vector $\Bv_s'$, which is a concatenation of vectors of form $(a_i \1/\sqrt{n})_{J_i'}$. We let $\Bv'=(\Bv_s', \Bv_m')$, where $\Bv_m' = \Bv_m$. Also, based on Subsection \ref{sub:approx}, without loss of generality we can assume the entries of $\Bv_m'$ are from $\Z/D$. 

Note that by definition,
\begin{equation}\label{eqn:Lv'}
\|(\wb{L}_{n}- \la_0) \Bv'\|_2 =O(\sqrt{ n} \log n \times r_n\sqrt{|I_s|}) ; \mbox{ and } |\sum_i v_i'| \le r_n |I_s|; \mbox{ and } \|\Bv'\|_2 = 1+ O(r_n \sqrt{I_s}),
\end{equation}
where $\sqrt{ n} \log n$ comes from the upper bound for $\|L-\la_0\|_2$. By the triangle inequality:
\begin{equation}\label{eqn:v'}
\|(\wb{L}_{n}- \la_0) \Bv'\|_2 \le \|(L- \la_0) \Bv\|_2+ \|(L- \la_0) (\Bv- \Bv')\|_2 =O(\sqrt{n} \log n \times n^{-C})+ O(\sqrt{ n } \log n \times r_n  \sqrt{|I_s|}).
\end{equation}
Although we will mention in more details later, as of this point we remark that our main starting point is to pass the event considered in Theorem \ref{thm:non-gap:lap} to the event that there exists $\Bv'$ such that \eqref{eqn:v'} holds.

As of this point, roughly speaking, we have extended the special vector $\Bv_{I}$ to a longest possible vector $\Bv_{s}$ whose entries take only a few values, while $\Bv_{m}$ is the left-over that we no longer detect entries with high multiplicities.

{\bf Case 1.} Let $\mathcal{E}_1$ be the event that there exists a non-zero $\Bv \in (\frac{1}{D}\Z)^n$ satisfying \eqref{eqn:Lv} and the resulting $I_m$ (from $\Bv$) is {\bf empty} in the sense that for all $i$ we can approximate $\Bv_{J_i}$. 

We let $\Bv'$ be obtained from $\Bv$ by replacing $\Bv_{J_i'}$ with $a_i \1/\sqrt{n}$, and together with some $v_j \in \Z/D$ for each $j \in J_i \bs J_i'$. As every sub-interval belongs to a level set, note that in this case we have 
$$\|\Bv_{J_i'} -a_i \1/\sqrt{n}\|_2 \le \delta_n t_{0}^{i_0} \le \delta_n t_{0}^{8\la_n^{-1}} .$$

The factor $t_0^{8\la_n^{-1}}$ can be as large as $(\log n)^{512 \la_n^{-1}}$, but as $\delta_n$ is sufficiently small from \eqref{eqn:delta_n}, the RHS is still small.

We will consider a $\delta_n t_{0}^{8\la_n^{-1}}$-approximation of the vectors. There are at most $D^{\al_n n} \times D^{8 \la_n^{-1}}$ choices of $\Bv'$ that can arise with empty $I_m$. The first factor stands for the amount of choices for the outliers; the second factor stands for the possible choices of the representative \( a_{i} \) of each interval, assuming that $1/D$ is small compared to $\delta_n t_{0}^{\la_n^{-1}}$, which is necessary for the approximation to make sense.

We write $\Bv'=(\Bv_1',\Bv_2')$ where $\Bv_1'$ are from the $(a_i/\sqrt{n})$ and the support of $\Bv_2'$ is at most $\al_n n$. It is important to emphasize that we are working with the event that there exists $\Bv'$ of that form such that $(L-\la_0)\Bv'$ is small. 

{\bf Subcase 1.1}. Assume first that the support of the $\{a_i, 1\le i \le \ell\}$ has length at least $2\delta_n'$, where
\begin{equation}\label{eqn:L:delta'}
\delta_n' := 8\la_n^{-3/2}  (\delta_n t_{0}^{\la_n^{-1}}) \sqrt{\log n}.
\end{equation}
Then either the smallest or largest values of $a_i$ is $\delta_n'$-far from $\la_n^{-1}/2$ of the rest. Assume without loss of generality that $a$, the multiplicity of the first segment, is $\delta_n'$-far from $\la_n^{-1}/2$ of the multiplicities $j_0$. 

Consider the event 
$$((\wb{L}_{n}- \la_0) \Bv')_i\le 2\sqrt{\log n}  (\delta_n t_{0}^{8\la_n^{-1}}), \,\mbox{where $i\in I_{j_0}$}.$$ 
This can be written as
\begin{align*}
    \sum_{j\neq i} \frac{a_j}{\sqrt{n}} x_{ij} - \frac{a_{i}}{\sqrt{n}} (\sum_{j\neq i} x_{ij})-(\la_0 \sum_{i}v_i') -f
    &= x_{i1}(a-a_{i})/\sqrt{n} + x_{i2}(a-a_{i})/\sqrt{n} + \dots\\
    &= (x_{i1} + x_{i2} + \dots + x_{i|J_1|})(a-a_{j_0})/\sqrt{n} + \dots,
\end{align*}
where $x_{ij}$ are the entries of the adjacency matrix, and $f$ depends on the $a_{i}$ but not on the $x_{ij}$\footnote{This $f$ comes from the part $\E L_{n}$ in the definition of $\wb{L}_{n}$.}.

Note that there is also the part for $\Bv_2'$, but we can fix the randomness over its vector support.

Hence, for all $i \in I_{j_0}$, because the sum $x_{i1} + x_{i2} + \dots + x_{i|J_1|}$ lies in $[\pm \Theta (\sqrt{\la_n n})]$ by the Central Limit Theorem,
$$\P\left(((\wb{L}_{n}-\la_0) \Bv')_i \le 2\sqrt{\log n}  \left(\delta_n t_{0}^{8\la_n^{-1}}\right)\right) \le 1/2,$$
provided that $ \delta_n' \sqrt{\la_n}$ is sufficiently larger than $\sqrt{\log n}  (\delta_n t_{0}^{\la_n^{-1}})$; this condition is satisfied by \eqref{eqn:L:delta'}.

Combining the above conclusion for all $i \in \cup_{j_{0}} I_{j_0}$ yields
$$\P\Big(((\wb{L}_{n}-\la_0) \Bv')_j \le 2\sqrt{\log n}  (\delta_n t_{0}^{8\la_n^{-1}}), i\in \cup_{j_{0}} I_{j_0}\Big) \le (1/2)^{(n/2- \al_n n)}.$$

Therefore, by tensorization (Lemma \ref{lem:tensorization}), we obtain
$$\P\Big(\|(\wb{L}_{n}-\la_0) \Bv'\|_2 =O(\sqrt{\log n}  (\delta_n t_{0}^{8\la_n^{-1}}) \sqrt{n})\Big) \le (1/2)^{n/3}.$$
Taking union bound over $D^{\al_n n} \times D^{8\la_n^{-1}} \le n^{C n/(\log n)^2}$ choices of $\Bv'$, we see that
$$
\P(\CE_1)\leq (1/2)^{n/4}.
$$

{\bf Subcase 1.2}. We next consider the case that all the values $a_i$ are of distance at most $2\delta_n'$. 
Then there exists $a$ such that
$$\|\Bv_1' - a \1/\sqrt{n} \|_2 \le O(\delta_n').$$ 
Recall that \( \Bv \) is almost orthogonal to \( \1 \), since \( \Bv \) is chosen as an non-trivial approximate eigenvector. Therefore, $\|\Bv - a \1/\sqrt{n}\|$ has order 1, so that $\|\Bv_2' -a \1/\sqrt{n})\|_{2}$ is of order 1. By triangle inequality,
$$\|(\wb{L}_{n}-\la_0) (\Bv' -a \1/\sqrt{n})\|_2 = O(\sqrt{\log n}  (\delta_n t_{0}^{\la_n^{-1}}) \sqrt{n}) + \sqrt{n} \log n (\delta_n') \ll \sqrt{n}/\log n,$$
provided that $\delta_n$ is chosen as in \eqref{eqn:delta_n}.

On the other hand, by the Levy anti-concentration bound (see for instance \cite[Lemma 2.6]{RV-LO}), there exists a constant $p<1$ such that 
$$\rho_{1/2}\left(\sum_{j \in \supp(\Bv_2')} x_{ij} (v_j'-a/\sqrt{n})\right)<p.$$
In particular, by tensorization, 
$$\P\Big((\wb{L}_{n}-\la_0) (\Bv' -a \1/\sqrt{n})=o( \sqrt{n})\Big) \le p^{n -2 \al_n n}.$$
Taking a union bound over the choices of $\Bv'$, we obtain $D^{\la_n^{-1}} D^{\al_n n} p^{n-2\al_n n}$, which is small as in Subcase 1.1.

Finally, taking union bound over $\binom{n}{\lfloor \la_{n} n \rfloor}$ choices for $I$, we see that, with $\la_{n}$ from \eqref{eqn:lambda}, the union bounds from both Subcase 1.1 and Subcase 1.2 are of order $e^{-\Theta(n)}$.
 
We now move to Case 2, which is more involved.

{\bf Case 2.} 
Let $\mathcal{E}_2$ be the event that there exists a non-zero $\Bv \in (\frac{1}{D}\Z)^n$, almost orthogonal to $\1$, whose first $\lceil \lambda_n n \rceil$ coordinates are the same,   $\|(\wb{L}_{n}-\la_0)\Bv\|_2 \le n^{-C}$, and the resulting $I_m$ (from $\Bv$) is not empty.  We have 
$$|I_m| \geq \lambda_n n/8 \mbox{ and } |I_s|\geq\lambda_n n.$$
Our treatment here is guided by the discussion from Subsection \ref{sub:overcrowding:discussion}. We first pass to the approximate eigenvector $\Bv'$ and consider the event \eqref{eqn:v'}. Assume that $I_s$ consists of $i_0\le 8\la_n^{-1}$ segments. We now describe a function $f$ from subsets of $[n]$ (that can occur as $I_s$) to subsets of $[n]$.  We describe $F=f(I_s)$, 
and write $I_m$ for $[n]\setminus I_s$, but note that $F$ does not depend
on $\Bv$. 
\begin{itemize} 
\item If $|I_s|>|I_m|$, we let $F$ be the first $|I_m|$ elements of $I_s$. 
\vskip .1in
\item Otherwise, we let $J_*$ be the first $J_i$ in $I_m$ and we let $F$ be the union of $I_s$ and the first $|I_m|-|I_s|$ elements of $I_m$ that are not in $J_*$. (This is possible because $|J_*|\leq \lambda_n n/2\leq  |I_s|.$) 
\end{itemize}

The following diagram explains the construction with $[-]$ denoting $I_m,I_s$ and $(-)$ denoting $F$:
\begin{figure}[h!]\centering
 
\begin{tikzpicture}
 
\draw  (0,2) node{[}node{(} -- (2, 2){}node[below]{$F$}  -- (4,2) node{)} node[above]{$I_s$}   -- (8,2)  node{]}node{[} -- (10,2) {} node[above]{$I_m$} -- (12,2) node{]};
\end{tikzpicture}
 
\caption{$|I_s|>|I_m|$}
\begin{tikzpicture}
 
\draw  (0,2) node{[}node{(} -- (2, 2){}node[above]{$I_s$}  -- (4,2) node{]}node{[}  -- (4.1,2) {} node[below]{$F$} -- (8,2)  node{)}node[above]{$I_m$}   -- (12,2) node{]};
\end{tikzpicture}
 
\caption{$|I_s|<|I_m|$} 
\end{figure}

In either case, we have 
$|F|=|I_m|$ and if $I_m$ is non-empty, it contains some $J_i$ that does not intersect $F$. Therefore, we create a square submatrix as well as an i.i.d.~thin submatrix out of these indices.

Let
\begin{equation}\label{eqn:c_n}
c^\ast_n = \log^{-2} n,
\end{equation}
and let $\CE_{drop}$ be the event that  there is a square submatrix $M_{A\times B}$ of $M=L_n$ of dimension $|A| = |B| \geq \lfloor n^{1 -\delta_0} \rfloor$ for a sufficiently small constant $\delta_0$ so that $n^{1-\delta_0} < \lambda_n n$, with the $c^\ast_n |A|$-least singular value at least $ c_0 c^\ast_n\sqrt{|A|}$.
Theorem \ref{thm:lap_overcrowding} tells us that
$
\P(\CE_{drop})\leq e^{O(n)} e^{-\Theta(\lfloor n^{1-\delta_0} \rfloor) \sqrt{n}}
$ 
(more specifically, $e^{c^\ast_n n \log n} e^{-\Theta(c^\ast_n\lfloor n^{1-\delta_0} \rfloor) \sqrt{n}}$), and thus
$
\P(\CE_{drop})\leq  e^{-c_{drop} n^{3/2-2\delta_0} },
$
for some $c_{drop}>0$.

We wish to bound the probability of $\mathcal{E}_2\setminus \CE_{drop}$.
Recall that for $\Bv$ causing $\mathcal{E}_2$, we have  $|I_m|> \la_n n/8$.

Outside of $\CE_{drop}$,  the square matrix $M_{F \times I_m}$ is near isometry: it has a rectangular submatrix $M_{12}=M_{F \times {{I_m'}}}$ of dimension $|I_m| \times (|I_m| -  c^\ast_n |I_m|) $  with the least (non-trivial) singular value at least $ c_0 c^\ast_n\sqrt{|I_{m}|}$. 

Given subsets $S$ (for $I_m$),  and $I_m' \subset I_m$ of size $|I_m'|=|I_m| - c^\ast_n |I_m|$, let $\mathcal{E}_{2.1}(S,{{{{I'_m}}}}))$ be the event that there exists $\Bv'$ such that the above holds, and $\|(\wb{L}_{n}-\la_0)\Bv'\|_2$ is  $O(\sqrt{ n} \log n \times r_n\sqrt{|I_s|})$. 

\begin{lemma}\label{claim:cond:fix} 
For all $S$ and \( I_{m}' \) such that $\mathcal{E}_{2.1}(S,I'_m)$ is defined, we have that
$$
\P(\mathcal{E}_{2.1}(S,{{{{I'_m}}}})) \leq \left( \frac{1}{t_0}\right)^{\lambda_n n/16}.$$
\end{lemma}
\begin{proof}[Proof of Lemma \ref{claim:cond:fix}] We fix a choice of $\Bv_s'$ and $\Bv_{I_m\setminus {{I'_m}}}'$. In what follows, we denote \( M_{11}, M_{12}, M_{21}, M_{22}\) as follows:
  $$M_{11}= L_{F \times  (I_s \cup (I_m\bs I_m')) }, M_{12}= L_{F \times I_m'}, M_{21}= L_{[n]\bs F \times  (I_s \cup (I_m\bs I_m'))}, M_{22} = L_{[n]\bs F \times I_m'}.$$
\[
\begin{bmatrix}
\multicolumn{3}{c|}{M_{11}} & & M_{12} \\
\multicolumn{3}{c|}{} & & \\
\hline
\multicolumn{3}{c|}{M_{21}} & & M_{22} \\
\end{bmatrix}
\]

Recall that the matrix $M_{12}$ is near isometry. Let $H_{12}$ be an $(|I_m|-c^\ast_n |I_m|) \times |I_m|$ matrix where $H_{12}M_{12}|_{I_m'} = I_{I_m'}$ (i.e. \( H_{12} \) is a left inverse of \( M_{12} \)). By the above (i.e. via the application of Theorem \ref{thm:lap_overcrowding}), we have 
$$\|H_{12}\|_{2} \le c_0^{-1}(c^\ast_n)^{-1}/\sqrt{|I_m|}.$$
If we write
$$\Bv_s'' = (\Bv_s', \Bv_{I_m\bs I_m'}') \quad \text{and} \quad \Bv_m'' = \Bv_{I_m'}',$$
so that 
\begin{equation}\label{eqn:1}
\|\wb{L}_{n} \Bv'\|_{2} = \|(M_{11}\Bv_s'', M_{12}\Bv_s'') + (M_{21}\Bv_m'', M_{22}\Bv_m'') \|_2 \le  \sqrt{ n }\log n \times r_n \sqrt{|I_s|}.
\end{equation}
Then we have that
$\|M_1' \Bv_s'' + M_2' \Bv_m''\|_{2} \le  \sqrt{ n} \log n \times r_n\sqrt{|I_s|}$.
In particular, after applying $H_{12}$ we have that 
$$\|H_{12}M_{11}\Bv_s''  +\Bv_m'' \|_2 \le  \sqrt{ n} \log n \times r_{n} c_0^{-1}(c^\ast_n)^{-1}\sqrt{|I_s|}/\sqrt{|I_m|} \le  \sqrt{n}\log n \times \la_n^{-1/2} r_n c_0^{-1}(c^\ast_n)^{-1}\left(\frac{1-c_m}{c_m}\right)^{1/2}$$ 
$$\le  \sqrt{n}\log n \times \la_n^{-1} r_n c_0^{-1}(c^\ast_n)^{-1},$$
where $|I_m| = c_m n$, with \( c_{m} \) satisfying $\la_n/8<c_m < 1-\la_n$. Here we note that $\delta_n$ is sufficiently small given all other parameters, and so the RHS bound will be small. 

In other words, $\Bv_m''$ is almost determined if we fix $M_{11}, M_{12}$ and $\Bv_s''$. At this point, we will replace $\Bv_m''$ by $\Bu_m'' = -H_{12}M_{11}\Bv_s''$. The latter vector is fixed if we condition on $M_{11}, M_{12}$, and $\Bv_s''$. Note that 
$$\|\Bv_m''- \Bu_m''\|_2 \le\sqrt{n}\log n \times \la_n^{-1} r_n c_0^{-1}(c^\ast_n)^{-1}=K_n r_n,$$
where 
$$K_n:= \sqrt{n}\log n \times \la_n^{-1} c_0^{-1}(c^\ast_n)^{-1}.$$ 
This approximation of $K_n r_n$ is rather big compared to $\sqrt{n} r_n$, but it is still smaller than $t_0^{1/2} \sqrt{n} r_n$ because $t_0$ is a much larger power of $\log n$.

We have 
\begin{align*}
\|M_{21} \Bv_s'' + M_{22} \Bu_m''\|_2 \le \|M_{21} \Bv_s'' + M_{22} \Bv_m''\|_2 + \|M_{22} (\Bv_m'' -\Bu_m'')\|_2 & \le  C \sqrt{ n }\log n \times r_n\sqrt{|I_s|} +  \|M_{22}\|_{2} K_n r_n \\
& \le C \sqrt{ n} \log n \times K_n r_n .
\end{align*}
 Note that $M_{21}$ is the matrix  $L_{[n]\bs F \times (I_s \cup (I_m\bs I_m'))}$ and $M_{22}$ is the matrix $L_{[n]\bs F \times I_m'}$. We let $J_*$ be one of the $J_i$ that is in $I_m'$ but has no intersection with $F$. It is over these coordinates that we will calculate the small ball probabilities and subsequently apply tensorization.

\begin{figure}[ht]
    \centering
\begin{tikzpicture}
 
\draw  (0,2) node{[}node{(} -- (2, 2){}node[below]{$F$}  -- (4,2) node{)} node[above]{$I_s$}   -- (8,2)  node{]}node{[} --(8.7,2) node{(} -- (9,2) node[below]{$J_*$} -- (9.3,2) node{)}
 -- (10,2) {} node[above]{$I_m' \subset I_m$} -- (12,2) node{]};
\end{tikzpicture}
 
\caption{$|I_s|>|I_m|$}
\begin{tikzpicture}
 
\draw  (0,2) node{[}node{(} -- (2, 2){}node[above]{$I_s$}  -- (4,2) node{]}node{[}  -- (4.1,2) {} node[below]{$F$} -- (8,2)  node{)}node[above]{$I_m' \subset I_m$} -- (8.7,2) node{(} -- (9,2) node[below]{$J_*$} -- (9.3,2) node{)}
  -- (12,2) node{]};
\end{tikzpicture}
 
\caption{$|I_s|<|I_m|$} 
\end{figure}

For each $i\in ([n]\setminus F) \setminus J_*$, we let $X_i$ be the $i$th row of $\wb{L}_n-\la_0$.  Then we can write $|X_i \cdot \Bv'| = O(\log n \times K_n r_n)$ as
\begin{equation}\label{E:toget}
|\sum_{j\in J_*} x_{ij}(v_j'-v_i') + \sum_{j\in [n]\setminus (J_* \cup \{ i\})} x_{ij}(v_j'-v_i')+ f| = O(\log n \times K_n r_n)  
\end{equation}
where $v_i'$ are the entries of $\Bv'$, and $f$ depends on the $v_{i}'$'s but not on the $x_{ij}$'s.

Notice that the collection $\left\{x_{ij} \colon i\in  ([n]\setminus F) \setminus J_*, j\in J_*\right\}$ is independent.
We further condition on $x_{ij}$ for $i\in  ([n]\setminus F) \setminus J_*$ and $j\not \in J_*$.
Since $J_*\cap F=\emptyset$, none of the $x_{ij}$ for $i\in ( [N]\setminus F) \setminus J_*$ and $j\in J_*$ have been conditioned on.
Thus, after our conditioning, the probability of \eqref{E:toget} holding is at most $\rho_{L,\log n \times K_n r_{n}}(\Bv_{J_*}').$
Thus, by tensorizing (Lemma \ref{lem:tensorization}) over the \( n - |F| - |J^{*}| \) independent rows, the probability that we have the event in the claim for a given $\Bv_s$ and $\Bv_{I_m\setminus {{I'_m}}}$ is at most
\begin{equation}\label{eq:sbp_tensor_case2}
\rho_{L,\log n \times K_n r_{n}}(\Bv_{J_*}')^{n-|F|-|J_*|}.
\end{equation}
Since $J_*$ is one of the $J_i$ that is a subset of $I_m$. Therefore, by the definition of non-sparse interval (\ref{eqn:dis}), for any $a$ there are at least $\al_n \la_n n$ coordinates $i\in J_*$ such that (noting that over the mix part, $v_i=v_i'$) 
$$|v_i'-a\1/\sqrt{n}| \ge t_0^{1/2} r_n.$$ 

Theorem~\ref{theorem:LO} (more specifically \eqref{eqn:pig}, and also in the spirit of Remark \ref{remark:L:sparse:rho}) implies that, with $r=t_0^{1/2}r_n$ and $R = \sqrt{\log n} K_n r_{n}$,
$$
\rho_{L,\log n \times K_n r_{n}}(\Bv_{J_*}') = O\left(\frac{ \log n \times K_n}{t_0^{1/2}} \frac{1}{\sqrt{\al_n \la_n n}}\right)= O\left(\frac{\log^{3/2} n \times \al_n^{-1/2} \la_n^{-3/2} c_0^{-1}(c^\ast_n)^{-1}}{ t_0^{1/2}}\right) = O\left(\frac{1}{t_0^{1/4}}\right).$$
Note that with the choices of parameters $\la_n$ from Theorem \ref{thm:non-gap:lap}, $c^\ast_n$ from \eqref{eqn:c_n}, $\al_n$ from \eqref{eqn:a_n}, and $t_0$ from \eqref{eqn:t_0}, we see the that above is much smaller than 1. We have $|J_*|\leq \lambda_n n/2$, $|F|=|I_{m}|$, therefore combining the above with \eqref{eq:sbp_tensor_case2}, we have that the probability of the event in the claim for a given $\Bv_s$ and $\Bv_{I_m\setminus {{I'_m}}}$ is at most
$$
\left(\frac{1}{t_0^{1/4}}\right)^{n-|F|-\lambda_n n/2} \le \left( \frac{1}{t_0^{1/4}}\right)^{|I_s|-\lambda_n n/2} \le \left( \frac{1}{t_0^{1/4}}\right)^{\lambda_n n/2}.
$$
The total number of $\Bv_s'$ and $\Bv_{I_m\setminus {{I'_m}}}'$ is bounded by $D^{i_0 \al_n \la_n n} \times D \sqrt{n}$ and $D^{c^\ast_n n}$, where the first bound comes from the choices of non-sparse elements times the cardinality of net from Remark \ref{remark:L:sparse:trivialnet}, while the second bound comes from $|I_m|-|I_m'| = c^\ast_n n$. Note that 
$$D^{i_0 \al_n \la_n n} \times D \sqrt{n}, D^{c^\ast_n n}  \le n^{O(\frac{n}{\log^2n})}  = e^{O(\frac{n}{\log n})}.$$
Hence we obtain the following upper bound for the event $\CE_{2.1}$ (given the choice of $t_0$ from \eqref{eqn:t_0}):
\begin{align*}
D^{i_0 \al_n \la_n n} \times D \sqrt{n} \times D^{c^\ast_n n} \times  \left( \frac{1}{t_0^{1/4}}\right)^{\lambda_n n/2} \le \left( \frac{1}{t_0}\right)^{\lambda_n n/16} .\end{align*}
\end{proof}

The event $\CE_2$ is the union of $\CE_{drop}$ with $\mathcal{E}_{2.1}(S,{{I'_m}},{{{{I'_s}}}})$ over $S,{{I'_m}},{{{{I'_s}}}}$, with $|S|-|{{I'_m}}|= c_\ast |I_m|$,
and $n-|S|\geq \lambda_n n$ and $|S|$ a possible value of $I_m$.  There are at most $2^{8\la_n^{-1}}$ possible values of $I_m$ (and hence $S$). Taking a union bound over all the choices, and over all $I$ (of Theorem \ref{thm:non-gap:lap}) of size $\la_n n$ we obtain an entropy bound at most $4^n$. We obtain the following bound for the event \eqref{eqn:v'} within Case 2
$$4^{n} \left( \frac{1}{t_0}\right)^{\lambda_n n/16} \le e^{-\Theta(n)},$$
given the choice of $t_0$ from \eqref{eqn:t_0} and $\la_n$ from Theorem \ref{thm:non-gap:lap}.

\section{No-Structure Delocalization: Proof of Theorem \ref{thm:LCD:lap} for the Approximate Eigenvectors}
We let $\CE_{non-gap}$ denote the event in Theorem \ref{thm:non-gap:lap}, where we have showed in the previous section that 
$$\P(\CE_{non-gap}) = 1- n^{{-\omega(1)}}.$$

We also let $\CE_{i.i.d.-norm}$ be the event that the matrix $\wb{L}_{n}$ with zeros on the diagonal (instead of the entries of $\diag(D_{n} - \E D_{n})$) has norm $O(\sqrt{n})$. It is well known that $\CE_{i.i.d.-norm}$ holds with probability $1 -e^{{\Theta(n)}}$. 

Throughout this section, if not specified otherwise, we will condition on these two overwhelming events $\CE_{non-gap}$ and $\CE_{i.i.d.-norm}$. 

Theorem \ref{thm:non-gap:lap} applied to $\la_{n}= 1/\log \log n$ yields that
\begin{corollary}\label{cor:lap_incompressible} Any subvector $\Bv_{J}$ of $\Bv$ of size $|J| =n_{0}/\sqrt{\log \log n}$ satisfies
$$\|\Bv_{J} -a1_{J}\|_{2} \ge \frac{1}{(\log n)^{A_{0} \log \log n}} =: c_{n}(=n^{{-o(1)}}).$$
\end{corollary}

Assume that there exists an approximate eigenvector $\Bv$ and an index $I$ such that 
\begin{equation}\label{eqn:v_{I}:start}
\LCD_{\kappa, \gamma}(\Bv_{I}) \in [D_{0},2D_{0}),
\end{equation}
where $D_{0} \asymp n^{A}$. Under $\CE_{non-gap}$, by Theorem \ref{thm:non-gap:lap} and Lemma \ref{lem:incomp_lcd}, we have that $A \ge 1/2$, for appropriate choices of $\kappa, \gamma$ (such as $\kappa =n^{1/3}$ and $\gamma=c_{n}^{3}$ -- which satisfy the conditions of Lemma \ref{lem:incomp_lcd} with $c_{0}$ a sufficiently small constant, and with $c_{1} = 1/(\log n)^{A_{0}/c_{0}}$ and $C=10$ respectively). 

The key idea to show that the event from \eqref{eqn:v_{I}:start} has small probability is as follows:

\begin{itemize}
\item  Step 1: If $\LCD(\Bv_{I})$ is small, $\Bv_{I}$ has structure. However, as we don't have any information on $\Bv_{I^{c}}$, we will use as a guide the discussion in Subsection \ref{sub:overcrowding:discussion} to solve for $\Bv_{I^{c}}$ by conditioning on some parts of the matrix $\wb{L}_{n}$. We will do so by relying on Theorem \ref{thm:lap_overcrowding}, see \eqref{eqn:v_{I^{c}}}.
\vskip .1in
\item Step 2: Starting from \eqref{eqn:v_{I^{c}}}, which roughly can be expressed as an event $\|M_{11} \Bv_{I} -\Bf\|_{2}$ being small for some fixed $\Bf$, we will work with the randomness of $M_{11}$, a principal minor of size $n_{0} \times n_{0}$.  As this matrix is symmetric, we have to decompose it into smaller rectangular blocks to see independent random rows and entries, see Figure \ref{figure:rec}. One major problem here is that the entries on the diagonal depend on other entries of the rectangular blocks, and hence we have to use the $\rho_{L}$ notion from Theorem \ref{theorem:LO} rather than its $\rho$ counterpart. Because there are many shifts to deal with, we will start from the shift giving the worst small ball probability (see \eqref{eqn:D:v_{I}}). This approach is significantly different from the regularized $\LCD$ structures used in \cite{V-symmetric} to deal with random symmetric matrices. 
\end{itemize}

Toward the second step, we will need to work with some more special form of $\Bv_{I}$, rather than knowing that its $\LCD$ is large. We first fix a subset $I$ of size $n_{0}$ in $[n]$. Divide $I$ into $k=\lfloor \sqrt{\log \log n} \rfloor$ intervals $I_{1},\dots, I_{k}$ (for instance, of indices arranged consecutively in $I$), each of size approximately 
$$m=n_{0}/k.$$ 
The vector $\Bv_{I}$ is decomposed into $\Bv_{I_{1}},\dots, \Bv_{I_{k}}$ accordingly.

Given $A$ and sufficiently large $C$ (which is also used in the approximations of Subsection \ref{sub:approx}), let 
$$\CA := \{a\in \R, |a| \le n^{-1/2} \log^{2}n, a = \frac{l}{n^{C+1}}, l\in \Z\}.$$ 
Further, let
$$\gamma' =\gamma/c_{n} = c_{n}^{2}, \mbox{ which has order $n^{{-o(1)}}$.}$$ 
Define \footnote{We note that $D(\Bv_{I})$ does depend on the partition sets $I_{1},\dots, I_{k}$ and on $\kappa, \gamma'$, but we suppress this dependency in the notation for simplicity.}
\begin{equation}\label{eqn:D:v_{I}}
D(\Bv_{I}):=\max_{j} \min_{a_{j}\in \CA} \LCD_{\kappa,\gamma'}\left(\frac{\Bv_{I_{j}}-a\1}{\|\Bv_{I_{j}}-a\1\|_{2}}\right)=\max_{j}  D_{j}.
\end{equation}

We first notice that, as $a_{j} \in \CA$ and $ c_{n} \le \|\Bv_{I_{j}}\|_{2} \le 1$, by Corollary \ref{cor:lap_incompressible}
\begin{equation}\label{eqn:v_{j}:bound} 
c_{n} \le \|\Bv_{I_{j}}-a_{j}\1\|_{2} \le \log^{4}n.
\end{equation}

\begin{claim}\label{claim:LCDast:1}
Assume that $\min_{a\in \CA} \LCD_{\kappa,\gamma'}\left(\frac{\Bv_{I_{j}}-a\1}{\|\Bv_{I_{j}}-a\1\|_{2}}\right)=D_{j}$, where $D_{j}=n^{{O(1)}}$. Then for any $a\in \CA$ and $r\ge 1/D_{j}$,
$$\P(|X \cdot (\Bv_{I_{j}}-a\1)|\le r)=O(r \log^{{O(1)}} n).$$
\end{claim}
\begin{proof} By definition, for any $a\in \CA$, we have $\LCD_{\kappa,\gamma'}\left(\frac{\Bv_{I_{j}}-a\1}{\|\Bv_{I_{j}}-a\1\|_{2}}\right) \ge D_{j}$. Hence by Theorem \ref{thm:small_ball_prob_LCD} (noting that under $\CE_{non-gap}$, \eqref{eqn:v_{j}:bound} holds), for any $r \ge 1/D_{j} \ge 1/\LCD_{\kappa,\gamma'}\left(\frac{\Bv_{I_{j}}-a\1}{\|\Bv_{I_{j}}-a\1\|_{2}}\right)$,
\begin{align*}
\sup_{b}\P(|X \cdot (\Bv_{I_{j}}-a\1) -b|\le r)&=\sup_{b}\P(|X \cdot \frac{\Bv_{I_{j}}-a\1}{\|\Bv_{I_{j}}-a\1}\|_{2} -b|\le \frac{r}{\|\Bv_{I_{j}}-a\1\|_{2}})\\
&= O(r (\gamma 'c_{n})^{-1}+ \exp(-\kappa^{2})) = O(r \log^{O(1)}n).
\end{align*}
\end{proof}

\begin{claim}
We have 
$$c_{n}^{-1}D(\Bv_{I}) \le \LCD_{\kappa,\gamma}(\Bv_{I}).$$ 
\end{claim}

\begin{proof} Assume $D(\Bv_{I}) =D$ is achieved at $\Bv_{I_{1}}$ with the corresponding shift $a_{1}\in \CA$. Then by Fact \ref{LCD:scaling}, as $1\le \|\Bv_{I}\|_{2}/\|\Bv_{I_{1}}\|_{2} \le c_{n}^{-1}$,
$$\LCD_{\kappa,\gamma}(\Bv_{I}) \ge c_{n}^{-1} \LCD_{\kappa,\gamma'}(\Bv_{1}) \ge c_{n}^{-1}\LCD_{\kappa,\gamma'}(\Bv_{1}-a_{1}).$$
\end{proof}

As a consequence of the above claim, as $c_{n}=n^{{-o(1)}}$, instead of \eqref{eqn:v_{I}:start} we will be working with the event (recalling \eqref{eqn:D:v_{I}})
\begin{equation}\label{eqn:v_{I}:start'}
D(\Bv_{I}) = \max_{j} D_{j} \in [D_{0},2D_{0}),
\end{equation}
for some $n^{1/2-o(1)} \le D_{0} \le n^{A}$.

Note that within this event, by definition we have every $D_{j} \ll D_{0}$, and hence we can $O(\kappa/D)$-approximate all the subvectors $\Bv_{j}-a_{j}\1$ (for some $a_{j}\in \CA$). Note that the norm of each of the vector $\Bv_{I_{j}}-a_{j}\1$ is at least $c_{n}$ by  \eqref{eqn:v_{j}:bound}, so by Lemma \ref{lemma:nets}, by gluing the nets of the subvectors together, we obtain a $O(\kappa k /D_{0})$-net $\CN_{I}$ of size $(C_{0}c_{n}^{-1}D_{0}/\sqrt{m})^{km} |\CA|^{k}$ for the vectors $\Bv_{I}$ where \eqref{eqn:v_{I}:start'} holds.

Now, for a fixed $\Bv_I$ in $\CN_{I}$, consider the event that there exists $\Bv_{I^{c}} \in \R^{n-n_{0}}$ (of norm of order 1, more specifically $\Bv_{I^{c}} \in \frac{1}{n^{C+1}} \Z^{n-n_{0}}$) so that 
\begin{equation}\label{eqn:LCD:eps}
\|(\wb{L}_{n} -\la_{0})\Bv\|_{2} \le \sqrt{n} \log n \times \kappa k /D_{0} =:\eps.
\end{equation}
If we let  $M_{11}, M_{12}, M_{21}$ and $M_{22}$ be the submatrices of row and columns indexed from $I$ and $I^{c}$ (for instance $M_{22}$ is a square matrix of size $|I^{c}|=(n-n_{0})$), then we can rewrite the above as
\begin{equation}\label{eqn:I,I^{c}:1}
\|(M_{11}\Bv_I, M_{21}\Bv_I) + (M_{12}\Bv_{I^{c}}, M_{22}\Bv_{I^{c}}) \|_{2} \le \eps.
\end{equation}
In particular, $\|M_{21} \Bv_I + M_{22} \Bv_{I^{c}}\|_{2}\le \eps$.

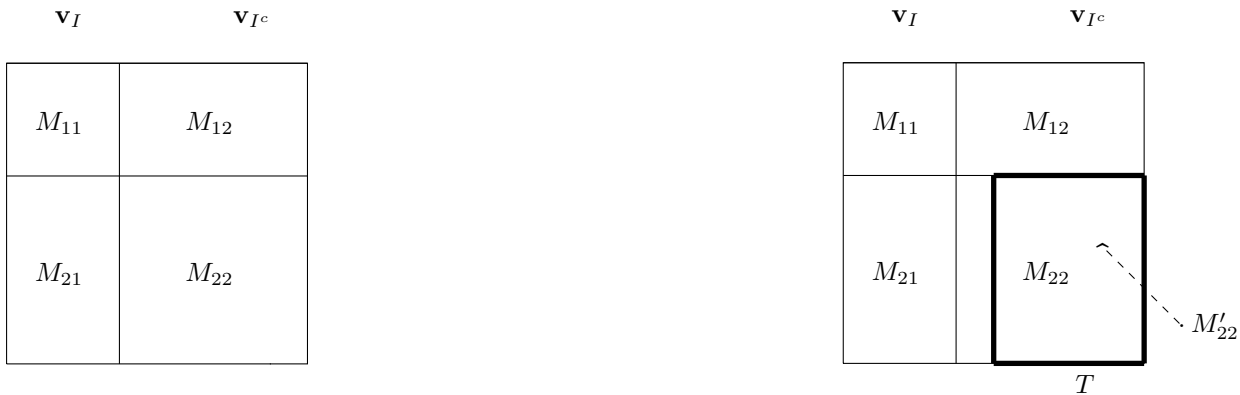
\begin{figure}
\centering

\begin{tikzpicture}
\draw (0,0) -- (4,0) -- (4,4) -- (0,4) -- (0,0);
\draw (1.5,0) -- (1.5,4);
\draw (0,2.5) -- (4,2.5);
\draw (0.7,3.2)  node{$M_{11}$};
\draw (2.7,3.2)  node{$M_{12}$};
\draw (3.5,0)--(3.5,0) node[below  left =.15in]{ };
\draw (0,4) -- (1.5,4)  node[above  left =.2in]{$\Bv_{I}$};
\draw (1.5,4) -- (4,4)  node[above  left =.2in]{$\Bv_{I^{c}}$};
\draw (2.7,1.2)  node{$M_{22}$};
\draw (0.7,1.2)  node{$M_{21}$};

\end{tikzpicture}
\hfill
\begin{tikzpicture}
\draw (0,0) -- (4,0) -- (4,4) -- (0,4) -- (0,0);
\draw (1.5,0) -- (1.5,4);
\draw (0,2.5) -- (4,2.5);
\draw (0.7,3.2)  node{$M_{11}$};
\draw (2.7,3.2)  node{$M_{12}$};
\draw [line width=2pt](2,0) -- (2,2.5);
\draw [line width=2pt](4,0) -- (4,2.5);
\draw [line width=2pt](2,2.5) -- (4,2.5);
\draw [line width=2pt] (2,0) -- (4,0);
\draw (3.5,0)--(3.5,0) node[below  left =.02in]{$T$};

\draw (0,4) -- (1.5,4)  node[above  left =.2in]{$\Bv_{I}$};
\draw (1.5,4) -- (4,4)  node[above  left =.2in]{$\Bv_{I^{c}}$};
\draw (2.7,1.2)  node{$M_{22}$};
\draw (0.7,1.2)  node{$M_{21}$};
\draw  [dashed] (3.5,1.5)   -- (4.5,0.5) node{.} node[right]{$M_{22}'$};
\draw (3.45,1.5) node{\textbf{\large \textasciicircum}};
\end{tikzpicture}

\caption{Finding non-structured subvectors.}
\label{figure:local}
\end{figure}

In what follows, let 
\begin{equation}\label{eqn:d_{n}}
k_0:=d_{n}(n-n_{0}), \mbox{ where } d_{n}= \left(\frac{1}{\log n}\right)^{2}.
\end{equation}

We next apply Theorem \ref{thm:lap_overcrowding} and Lemma \ref{lem:lowerbound} to the matrix $M_{22}$: with probability at least $1- e^{- \Theta(k_0^{3/2}/\log n)}$, there exists $T \subset I^{c}$ of size $|T| =n-n_{0}-k_{0}$ such that the matrix $M_{22}|_{T}$ is near isometry in the sense that 
$$\sigma_{n-n_{0}-k_{0}}(M_{22}|_{T}) \gg d_{n}\sqrt{n-n_{0}}.$$
Let $M_{22}'$ be an $(n-n_{0}-k_{0}) \times (n-n_{0})$ matrix where $M_{22}' M_{22}|_T = I_{n-n_{0}-k_{0}}$ (i.e. left inverse). By the above, we can assume 
$$\|M_{22}'\|_{2} =O(1/d_{n}\sqrt{n-n_{0}}) = O(1/d_{n}\sqrt{n}).$$
Then, from the equation $\|M_{21}\Bv_I + M_{22}\Bv_{I^{c}} \|_{2} \le \eps$ coming from \eqref{eqn:I,I^{c}:1}, we apply $M_{22}'$, which yields
$$\|M_{22}'M_{21}\Bv_I + M_{22}' M_{22}|_{T^c} \Bv_{I^{c}}|_{T^c} +\Bv_{I^{c}}|_T \|_{2} \le \eps \|M_{22}'\|_{2} \le \eps /d_{n}\sqrt{n}.$$
In other words, if we fix a realization of $M_{21}$ (and hence $M_{12}$) and $M_{22}$, and if we fix $T$ and a choice of $\Bv_{I^{c}}|_{T^c}$, then $\Bv_{I^{c}}|_T$ lies within a distance of at most $\eps/d_{n}\sqrt{n}$ to $-M_{22}'M_{21}\Bv_I - M_{22}' M_{22}|_{T^c} \Bv_{I^{c}}|_{T^c}$.
Using this approximation in our other equation yields
$$\Big\|M_{12}\Bv_{I^{c}} -  M_{12}'\Bv_{I^{c}}|_{T^{c}} -  M_{12}'' (M_{22}'M_{21}\Bv_I + M_{22}' M_{22}|_{T^c} \Bv_{I^{c}}|_{T^c})\Big\|_{2} \le  \|M_{12}\|_{2} (\eps /\sqrt{n}) =O(  \eps/d_{n}),$$
where we recall that under $\CE_{i.i.d.-norm}$, $\|M_{12}\|_{2} = O(\sqrt{n})$. Together with the first half of \eqref{eqn:I,I^{c}:1}, $\|M_{11} \Bv_I + M_{12} \Bv_{I^{c}}\|_{2}\le \eps$, by the triangle inequality we have
\begin{equation}\label{eqn:v_{I^{c}}}
\Big \|M_{11}\Bv_I -M_{12}'\Bv_{I^{c}}|_{T^{c}} -  M_{12}'' (M_{22}'M_{21}\Bv_I + M_{22}' M_{22}|_{T^c} \Bv_{I^{c}}|_{T^c}) \Big \|_{2} =O( \eps/d_{n}). 
\end{equation}
Note that $M_{11}$ has size $n_{0} \times n_{0}$, and recall the definition of $\eps$ from \eqref{eqn:LCD:eps}. We will bound the probability with respect to the off-diagonal entries of $M_{11}$, while other entries of $\wb{L}_{n}$ are held fix. More specifically, we can write the above as 
$$\|(M_{11}-F)\Bv_{I}\|_{2} = O((\sqrt{n} \log n) \kappa k/D_{0} d_{n}) =O(\sqrt{n} \rho),$$ 
where $F$ is a deterministic matrix, and where for brevity we write
$$\rho:= (d_{n}^{-1} \log n \times \kappa k)/D_{0}.$$

\begin{figure}

\centering

\begin{tikzpicture}
\draw (0,0) -- (4,0) -- (4,4) -- (0,4) -- (0,0);
\draw (0,1) -- (4,1);
\draw (3,0) -- (3,4);
\draw (0,4) -- (4,0);
\draw (2,2) node[left=0.1in]{$ \sum_{j;j\neq i} x_{ij}-\la_{0}$} ;
\draw (2,4.5) node[above]{$v_{i}$} ;

\foreach \Point in {(2,2),(2,4.5)}{
    \node at \Point {\textbullet};
};

\draw (0,4) -- (1.5,4) node[above]{$I\bs I_{j_{0}}$};
\draw (3,4) -- (3.5,4) node[above]{$I_{j_{0}}$};
\draw [fill=black] (3,1) -- (4,1) -- (4,4) -- (3,4);
\draw (4,1)--(4,2.5) node[right=0.1in]{$L_{(I\bs I_{j_{0}}) \times I_{j_{0}}}$}; 
\end{tikzpicture}
\caption{Decomposition into i.i.d.~parts.}
\label{figure:rec}
\end{figure}
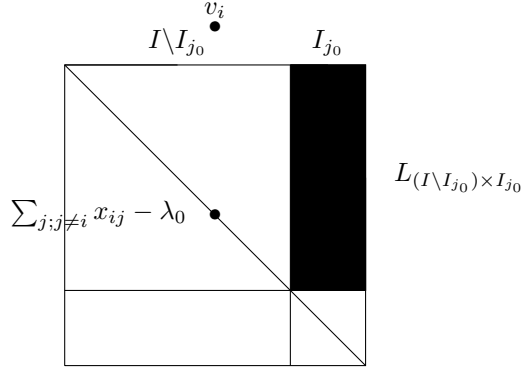

 Assume that $D(v_{I})$ is achieved at the interval $I_{j_{0}}$. Then by the rectangular decomposition trick (see Figure \ref{figure:rec}), we can pass to bounding the event  
 $$\|(M_{11})_{(I \bs I_{j_{0}}) \times I_{j_{0}}} \Bv_{I_{j_{0}}} - \Bw\|_{2} =O( \sqrt{n}\rho)$$ for some deterministic vector $\Bw$, where the randomness now is from the entries of $(M_{11})_{(I \bs I_{j_{0}})}$, which are all independent. 
 
 Next, for each $i\in I \bs I_{j_{0}}$ where $|v_{i}| \le n^{{-1/2}} \log^{2}n$, we consider the event $|\row_{i}(M_{11})_{(I \bs I_{j_{0}}) \times I_{j_{0}}} \cdot \Bv_{I_{j_{0}}} -w_{i}| =O(\rho) $, which has the form 
 $$|x_{1} (v_{i_{1}}-v_{i})+\dots+x_{m}(v_{i_{m}}-v_{i}) -f| = O(\rho), $$ 
where $x_{1},\dots, x_{m}$ are independent Bernoulli of parameter $p$, and where we recall that $km=n_{0}$ (with $k =\lfloor \sqrt{\log \log n}\rfloor$ and $n/k \le n_{0} \le n$), $i_{1},\dots, i_{m}$ are the elements of $I_{j_{0}}$, and $f$ might depend on $v_{i_{j}}$ but not on the $x_{i}$. 
 
We note that as $\sum_{i\in I_{j_{0}}} v_{i}^{2} \le 1$, the set indices $i$ where $|v_{i}| \le n^{{-1/2}} \log^{2}n$ has cardinality at least $n_{0}-m-n/\log^{4}n$. By Claim \ref{claim:LCDast:1}, as $D_{j_{0}} = D(\Bv_{I}) \in [D_{0}, 2D_{0}]$ and as $\rho \ge 1/D_{0}$,
$$\P\Big(|x_{1} (v_{i_{1}}-v_{i})+\dots+x_{m}(v_{i_{m}}-v_{i}) -f| =O(\rho)\Big)= O(\rho \log^{{O(1)}}n)).$$
Therefore, by Lemma \ref{lem:tensorization}, we have
$$\P\Big(\|(M_{11})_{(I \bs I_{j_{0}}) \times I_{j_{0}}} \Bv_{I_{j_{0}}} - \Bw\|_{2} =O( \sqrt{n}\rho\Big)  \le (\rho \log^{{O(1)}}n)^{n_{0}-m-n/\log^{4}n}.$$

Summing over $(C_{0}c_{n}^{-1}D_{0}/\sqrt{m})^{n_{0}} |\CA|^{k}$ choices for the vector $\Bv_{I}$, over $(n^{C+1})^{d_{n} (n-n_{0})} \times \binom{n-n_{0}}{d_{n}(n-n_{0})}$ choices for the vector $\Bv_{I^{c}}|_{T^c}$, and over $\binom{n}{n_{0}}$ choices for $I$, we obtain the final union bound
\begin{align*}
&(\rho \log^{{O(1)}}n)^{n_{0}-m-n/\log^{4}n} (C_{0}c_{n}^{-1}D_{0}/\sqrt{m})^{n_{0}} |\CA|^{k} (n^{C+1})^{d_{n} (n-n_{0})} \binom{n-n_{0}}{d_{n}(n-n_{0})} \binom{n-n_{0}}{d_{n}(n-n_{0})} \binom{n}{n_{0}} \\
\le & (\kappa \log^{O(1)}n/D_{0})^{n_{0}-m-m/\log^{2C}n} \times (n^{o(1)}D_{0}/\sqrt{m})^{n_{0}} \times n^{n/\log n} \times (C')^{n} = n^{-\Theta(n_{0})},
\end{align*}
provided that $\kappa =n^{c}$ with a sufficiently small constant $c$.

\section{Simplicity of the Spectrum: Proof of Theorems \ref{thm:gap:lap'} and \ref{thm:gap:lap}}\label{section:gap}

We first work with the centered model. 

\begin{proof}[Proof of Theorem \ref{thm:gap:lap'}] 
Let $\wb{L}_{n}'$ be obtained from $\wb{L}_{n}$ by changing the last two columns (and rows respectively) $\col_{n-1},\col_{n})$ of $\wb{L}_{n}$ to $(\col_{n-1}+\col_{n})/\sqrt{2}$ and $(-\col_{n-1}+\col_{n})/\sqrt{2}$. We observe the followings.

\begin{lemma} The matrix $\wb{L}_{n}'$ has the same spectrum as of $\wb{L}_{n}$.
\end{lemma} 
\begin{proof} Let 
\[
U_2 =\begin{pmatrix} 1/\sqrt{2} & -1/\sqrt{2} \\ 1/\sqrt{2} &1/\sqrt{2} \end{pmatrix}.
\] 
Note that $\begin{pmatrix} I_{n-2}& 0 \\ 0 &  U_2 \end{pmatrix}$ is orthogonal, and so 
\[
\wb{L}_{n}'= \begin{pmatrix} I_{n-2}& 0 \\ 0 &  U_2 \end{pmatrix}^T \wb{L}_{n}\begin{pmatrix} I_{n-2}& 0 \\ 0 &  U_2 \end{pmatrix}
\]
has the same spectrum as $\wb{L}_{n}$.  
\end{proof}
Therefore, it suffices to prove Theorem \ref{thm:gap:lap'} for $\wb{L}'_{n}$.  

Our next lemma is the following. 
\begin{lemma}\label{lemma:L_n-1_prime} The conclusion of Theorem \ref{thm:non-gap:lap} and \ref{thm:LCD:lap} holds for the eigenvectors of all principals $\wb{L}'_{n-1}$ of $\wb{L}'_{n}$ of size $(n-1) \times (n-1)$.
\end{lemma}
\begin{proof} The proof of this result is identical to the proofs of Theorems \ref{thm:non-gap:lap} and \ref{thm:LCD:lap}. Therefore, we omit the details.
\end{proof}
Let $\Bu = \left(\begin{array}{c} \Bw \\ b\end{array} \right)$ be an eigenvector of $\wb{L}'_{n}$ with associated eigenvalue $\lambda_i(\wb{L}'_{n})$, where \( \Bw \in \R^{n-1} \).  By Theorem \ref{thm:non-gap:lap} and Claim \ref{claim:non-gap:spread}, we assume that $|b| \ge n^{{-1/2-o(1)}}$ (as these results imply that with probability $1-n^{-\omega(1)}$ the set of indices where $|b| \ge n^{{-1/2-o(1)}}$ is of size $\Theta(n)$).

We consider the decomposition 
\[
\wb{L}'_{n} = \left( \begin{array}{cc}
    \wb{L}'_{n-1} & \col_{n}(\wb{L}'_{n}) \\
    \col_{n}(\wb{L}'_{n})^T & f \end{array} \right),
\]
where $\col_{n}(\wb{L}'_{n})$ is a $(n-1) \times 1$ vector \footnote{We slightly abused our standard notation, here $\col_{n}$ is not exactly the last column of $\wb{L}'_{n}$ but only a subvector.} and $f \in \R$. Arguing as in the discussion following Theorem \ref{thm:gap:lap'}, let $\Bv$ be a unit vector of $\wb{L}'_{n-1}$ corresponding to $\la_{i}(\wb{L}'_{n-1})$. We have
$$|\Bv \cdot \col_{n}(\wb{L}'_{n})| \le n^{{o(1)}}\delta.$$

Let $I_{n-2}$ be the set of vertices $v_1,\dots, v_{n-2}$ which has exactly one neighbor with $v_n$ and $v_{n-1}$. Then the neighbor switching process changes the coordinates over $I_{n-2}$ of $\col_{n}$ randomly and independently. Furthermore, with probability $1- \exp(-\Theta(n))$, we know that 
$$ n/4 - n/8\le |I_{n-2}| \le n/4 + n/8.$$
Conditioned on the event of probability $1-n^{-\omega(1)}$ from Theorem \ref{thm:LCD:lap} (applied to $\wb{L}'_{n-1}$) that $\LCD_{\kappa,\gamma}(\Bv_{I_{n-2}}) \ge n^{2A}$, by Theorem \ref{thm:small_ball_prob_LCD}, for any $\delta \ge n^{-A}$ we have that 
$$\P_{\col_{n}(\wb{L}'_{n})}(|\Bv \cdot \col_{n}(\wb{L}'_{n})| \le n^{o(1)}\delta) \le \sup_{b} \P_{\col_{n}(\wb{L}'_{n})}(|\Bv_{I_{n-2}} \cdot \col_{n}(\wb{L}'_{n})_{I_{n-2}}-b| \le n^{o(1)}\delta)=O( n^{o(1)}\delta).$$
\end{proof}

We next conclude with a proof of the non-centered model, where we recall that in our notation, the real eigenvalues of a symmetric matrix $M_{n}$ are arranged as $\la_{1}(M_{n}) \le \dots \le \la_{n}(M_{n})$.

\begin{proof}[Proof of  Theorem \ref{thm:gap:lap}]
    Now that we have established the result for $\wb{L}_{n}$, we translate the result back to $L_{n}$. Assume that the spectral decomposition of $\wb{L}_{n}$ is (with $\la_{1}\le \dots \le \la_{n}$)
    $$\wb{L}_{n} = \sum_{i} \la_{i} \Bv_{i}^{T} \Bv_{i},$$
    where one of the eigenvalues is $\la_{i_{0}}=0$, with $\Bv_{i_{0}} = \wb{\1} = \1/\sqrt{n}$, and where the eigenvectors $\Bv_{i}$ form an orthonormal basis. Note that we have given effective gaps between the $\la_{i}$.
    
      Observe that $\E L_{n} = -p J_{n} + pn I_{n}$, where $J_{n}$ is the $n \times n$ matrix of ones. So 
      $$L_{n} = \wb{L}_{n} + \E L_{n}= \sum_{i} \la_{i} \Bv_{i}^{T} \Bv_{i} - p J_{n} +  pn I_{n} = \sum_{i \neq i_{0}} (\la_{i}+pn) \Bv_{i}^{T} \Bv_{i}  + (pn - pn)\wb{\1}^{T}\wb{\1}.$$
Since $L_{n}$ is positive semidefinite, the zero eigenvalue of $L_{n}$ is the smallest eigenvalue $\la_{1}(L_{n})$, with eigenvector $\wb{\1}$ \footnote{Also, $L_{n}$ and $\wb{L}_{n}$ have the same set of eigenvectors.}. So for Theorem \ref{thm:gap:lap}, it suffices to bound the gap of the two smallest eigenvalues of $L_{n}$, $\la_{2}(L_{n}) - \la_{1}(L_{n}) = \la_{2}(L_{n})$.

Note that $L_{n} = D_{n}-A_{n}$.  We can view $A_n$ as a perturbation of a $\E A_n = p J - p I$, which has eigenvalues $-p$ with multiplicity $n-1$ and $(p-1) n$ with multiplicity one.  By classical bounds on the norm of a Wigner matrix, we have that $\|A_{n} - \E A_{n}\| = O_p(\sqrt{n})$ with probability $1 - \exp(-\Theta(n))$.  Therefore, by Weyl's inequalities,
\[
\lambda_n(A_{n}) - \lambda_{n-1}(A_{n}) \geq pn/2
\] 
with probability at least $1-\exp(-\Theta(n))$.  Thus, on this event
    \[
    \lambda_{2}(\E D_{n} - A_{n}) - \lambda_{1}(\E D_{n} - A_{n}) \geq pn/2.
    \]
 Finally, by Weyl's bound, the spectral gap $\la_{2}(L_{n}) - \la_{1}(L_{n})$  of $L_{n}$ is then at least 
    \[
    pn/2 - 2 \|D_{n} - \E D_{n}\|.
    \]
    Chernoff's bound and a simple union bound tells us that $\|D_{n} - \E D_{n}\|_{2} \leq pn/10$ with probability at least $1 - \exp(-\Theta(n))$.  Therefore, with probability at least $1 - \exp(-\Theta(n))$, $\lambda_{2}(L_{n}) \geq pn/3$, which is well above the gap size and probability proved for $L_{n}$.

\end{proof}

\section{Proof of Theorem \ref{thm:smallcoordinates}}
\begin{proof}(of  Theorem \ref{thm:smallcoordinates}) As $\wb{L}_n$ and $L_{n}$ have the same set of eigenvectors, it suffices to work with $\wb{L}_{n}$. We consider the decomposition 
\[
\wb{L}_{n} = \left( \begin{array}{ccc}
	\wb{L}_{n-2} & \col_{n-1}(\wb{L}_{n}) & \col_{n}(\wb{L}_{n}) \\
	\col_{n-1}(\wb{L}_{n})^T & f & g \\
	\col_{n}(\wb{L}_{n})^T & g & h 
\end{array} \right),
\]
where $\col_{n-1}(\wb{L}_{n})$ and $\col_{n}(\wb{L}_{n})$ are $(n-2) \times 1$ vectors \footnote{As in the previous section, here $\col_{n-1}, \col_{n}$ are not exactly the last two columns of $\wb{L}_{n}$.} and $f, g, h \in \R$. We slightly modify the argument following Theorem \ref{thm:gap:lap'}.
We consider the eigenvalue equation,
\begin{equation} \label{eq:n-2minor}
\left( \begin{array}{ccc}
	\wb{L}_{n-2} & \col_{n-1}(\wb{L}_{n}) & \col_{n}(\wb{L}_{n}) \\
	\col_{n-1}(\wb{L}_{n})^T & f & g \\
	\col_{n}(\wb{L}_{n})^T & g & h 
\end{array} \right) \begin{pmatrix} \Bw \\ a\\ b \end{pmatrix} = \lambda_i(\wb{L}_n) \begin{pmatrix} \Bw \\ a\\ b \end{pmatrix}
\end{equation}
with $|a|, |b| \leq n^{-B}$ and $\Bw \in \R^{n-2}$. If we extract the top $n-2$ coordinates, this implies that
\[
\wb{L}_{n-2} \Bw + a \col_{n-1}(\wb{L}_{n}) + b \col_{n}(L_{n}) = \lambda_i(\wb{L}_n) \Bw.
\] 
For $B$ large enough, since we can control the size of $\col_{n-1}(\wb{L}_{n})$ and $\col_{n}(\wb{L}_{n})$, this implies that
\[
\|(\wb{L}_{n-2} - \lambda_i(\wb{L}_{n})) \Bw\|_2 \leq n^{-B/2}. 
\]
In other words, $\Bw$ is an approximate eigenvector of $\wb{L}_{n-2}$.

Additionally, the final coordinate of \eqref{eq:n-2minor} yields
\[
|\col_{n}(\wb{L}_{n})^T \Bw| \leq |ga| + |hb| + |\lambda_i(\wb{L}_n)||b|,
\] 
which for large enough $B$ implies that
\[
|\col_{n}(\wb{L}_{n})^T \Bw| \leq n^{-B/2}.
\]
The rest of the proof is almost identical to the proof of Theorem \ref{thm:gap:lap'}. Let $I_{n-2}$ be the set of vertices $v_1,\dots, v_{n-2}$ which has exactly one neighbor with $v_n$ and $v_{n-1}$. Then the neighbor switching process changes the coordinates over $I_{n-2}$ of $\col_{n}(\wb{L}_{n})$ randomly and independently. Furthermore, with probability $1- \exp(-\Theta(n))$, we know that 
$$ n/4 - n/8\le |I_{n-2}| \le n/4 + n/8.$$
Conditioned on the event of probability $1-n^{-\omega(1)}$ from Theorem \ref{thm:approx:eigenvector} (applied to $\wb{L}_{n-2}$ and with sufficiently large $B$) that $\LCD_{\kappa,\gamma}(\Bw_{I_{n-2}}) \ge n^{2A}$, by Theorem \ref{thm:small_ball_prob_LCD}, for $\delta = n^{-A-3}$ we have that 
\begin{align*}
	\P_{\col_{n}(\wb{L}_{n})}(|\Bw \cdot \col_{n}(\wb{L}_{n-2})| \le n^{-B/2}) &\leq \P_{\col_{n}(\wb{L}_{n})}(|\Bw \cdot \col_{n}(\wb{L}_{n})| \le n^{o(1)}\delta) \\
	&\le \sup_{b} \P_{\col_{n}(\wb{L}_{n})}(|\Bw_{I_{n-2}} \cdot \col_{n}(\wb{L_{n}})_{I_{n-2}}-b| \le n^{o(1)}\delta)\\
	&=O( n^{o(1)}\delta).\end{align*}
We can then take a union bound over the $\binom{n}{2}$ possible sets of pairs of coordinates.  
\end{proof}

\appendix

\bibliographystyle{plain}
\bibliography{lap_bib.bib}

\end{document}